\documentclass[a4paper]{amsart}
\usepackage[english]{babel}
\usepackage{epsfig}

\usepackage{amsmath}
\usepackage{amsfonts}
\usepackage{amssymb}
\usepackage{amsthm}
\usepackage{epic}
\usepackage{fancyhdr}
\usepackage{fancybox}
\usepackage{subfigure}

\makeatletter
\renewenvironment{proof}[1][\proofname]{\par
  \pushQED{\qed}%
  \normalfont \topsep6\p@\@plus6\p@\relax
  \trivlist
  \item[\hskip\labelsep
        \textsc
    #1\@addpunct{.}]\ignorespaces
}{%
  \popQED\endtrivlist\@endpefalse
}
\makeatother

\usepackage{url}

\usepackage{amsopn}
\usepackage{latexsym}

\usepackage[pagebackref=true]{hyperref}
\hypersetup{   pdfborder={0 0 0},   
              colorlinks = false}

\renewcommand*{\backref}[1]{}
\renewcommand*{\backrefalt}[4]{%
  \ifcase #1 %
    No citations.
  \or
    (Page #4).%
  \else
    (Pages #4).%
  \fi%
}

\newcommand{\R}{\mathbb R}
\newcommand{\Z}{\mathbb Z}
\newcommand{\Q}{\mathbb Q}
\newcommand{\N}{\mathbb N}

\newtheorem{Theorem}[equation]{Theorem}
\newtheorem{Proposition}[equation]{Proposition}
\newtheorem{Lemma}[equation]{Lemma}

\newtheorem{Corollary}[equation]{Corollary}

\theoremstyle{definition}
\newtheorem{Definition}[equation]{Definition}
\theoremstyle{remark}
\newtheorem{Rmk}[equation]{Remark}
\newtheorem{Ex}[equation]{Example}

\title{Approximation Results for \texorpdfstring{$\alpha $}{alpha}-Rosen Fractions}
\author{Cor Kraaikamp} 
\address{Delft University of Technology and Thomas Stieltjes Institute for
Mathematics, DIAM, Mekelweg 4, 2628 CD Delft, the Netherlands}
\email{c.kraaikamp@tudelft.nl}
\author{Ionica Smeets}
\address{Universiteit Leiden and Thomas Stieltjes Institute for
Mathematics, Niels Bohrweg 1, 2333 CA Leiden, the Netherlands}
\email{ionica.smeets@gmail.com}
\date{\today}


\begin{document}
\maketitle

\begin{abstract}
In this article we generalize Borel's classical approximation results for the regular continued fraction expansion to the $\alpha$-Rosen fraction expansion,  using a geometric method. We give a Haas-Series-type result about all possible good approximations for the $\alpha$ for which the Legendre constant is larger than the Hurwitz constant.
\end{abstract}

\section{Introduction}

In 1798 Legendre proved the following result~\cite{L}.

\begin{Theorem}
\label{th: legendre}
Let $x\in \R \setminus \Q$ and $\frac{p_n}{q_n}$ be the $n$th regular continued fraction convergent of $x$, $n\geq 0$.  If $p, q\in \Z$, $q>0$, and $\gcd (p,q)=1$, then
$$
\left| x-\frac{p}{q}\right| < \frac{1}{2q^2}\quad
\textrm{ implies that}\quad \left[
\begin{array}{c}
p \\
q
\end{array}\right] = \left[ \begin{array}{c}
p_n\\
q_n
\end{array}\right] ,\quad \text{for some $n\geq 0$}.
$$
\end{Theorem}

Legendre's Theorem is one of the main reasons for studying continued fractions, because it tells us that good approximations of irrational numbers by rational numbers are given by continued fraction convergents. We call the best possible coefficient of $\frac{1}{q^2}$, independent of $x$, the Legendre constant. It is $\frac12$ for regular continued fraction (RCF) expansions. For the nearest integer continued fraction expansion (NICF) the Legendre constant is $g^2$, where $g = \frac{\sqrt{5}-1}{2} \approx 0.61$ is the golden number; see~\cite{Ito}.

\begin{Definition}
\label{def: intro theta}
Let $x\in [0,1)\setminus \Q$ and $\frac{p_n}{q_n} $ be the $n$th regular continued fraction convergent of $x$, $n\geq 0$. The approximation coefficient $\Theta_n~=~\Theta_n (x)$ is
defined by
\[
\Theta_n=q_n^2 \left| x - \frac{p_n}{q_n} \right| , \textrm{ for } n \geq 0.
\]
 \end{Definition}

We usually suppress the dependence of $\Theta_n$ on $x$ in our notation. The approximation coefficient gives a numerical indication of the quality of the approximation; for the RCF it easily follows that $\Theta_n \leq 1$. For the RCF-expansion we have the following classical theorems by Borel (1905)~\cite{B} and Hurwitz (1891)~\cite{H} about the quality of the approximations.
\begin{Theorem} 
\label{Borel}For every irrational
number $x$, and every $n\geq 1$
$$
\min\{\Theta_{n-1},\Theta_{n},\Theta_{n+1}\} <
\displaystyle{\frac{1}{\sqrt{5}}}.
$$
The constant $1/\sqrt{5}$ is best possible.
\end{Theorem}

\begin{Rmk}
If we replace $\frac{1}{\sqrt{5}}$ by a smaller constant $C$, then there are countably infinitely many irrational numbers $x$ for which
\[
\left| x -\frac{p}{q} \right| \leq \frac{C}{q^2}
\]
holds for only finitely many pairs of integers $p$ and $q$. An example of such a number is the small golden number $g$.
 \end{Rmk}

Because $ \frac{1}{\sqrt{5}}< \frac12 $ the results of Legendre and Borel imply Hurwitz's Theorem  that states that for every irrational number $x$ there exist infinitely many pairs of
integers $p$ and $q$, such that
$$
\left| x -\frac{p}{q}\right| < \frac{1}{\sqrt{5}}\frac{1}{q^2}
$$
We call the best possible coefficient of $\frac{1}{q^2}$ in this inequality  the Hurwitz constant. It is $\frac{1}{\sqrt{5}}$ for RCF-expansions. 

J.C.~Tong \cite{T1,T2} generalized Borel's result for the nearest integer continued fraction expansion (NICF). He showed that for the NICF there exists a `spectrum,' i.e., there exists a sequence of constants $(c_k)_{k\geq 1}$, monotonically decreasing to $1/\sqrt{5}$, such that for all irrational numbers $x$ the minimum of any block of $k+2$ consecutive NICF-approximation coefficients is smaller than $c_k$.

\begin{Theorem}For every irrational number $x$
and all positive integers $n$ and $k$ one has
$$
\min\{\Theta_{n-1},\Theta_{n},\ldots,\Theta_{n+k}\} <
\frac{1}{\sqrt{5}} +
\frac{1}{\sqrt{5}}\left(\frac{3-\sqrt{5}}{2}\right)^{2k+3}.
$$
The constant $c_k=\frac{1}{\sqrt{5}} +
\frac{1}{\sqrt{5}}\left(\frac{3-\sqrt{5}}{2}\right)^{2k+3}$ is
best possible.
\end{Theorem}

In~\cite{HK} Hartono and Kraaikamp showed how Tong's result follows by a geometrical method based on the natural extension of the NICF. The method will be discussed in~Section~\ref{sec:geo}. In~\cite{KSS} this method was  extended to Rosen fractions, yielding the next theorem.

\begin{Theorem}
\label{th: tong even} Fix $q = 2p$, with $p \in \N, p \ge 2$ and let $\lambda=\lambda_q=2 \cos\frac{\pi}{q}$.  For
every $G_q$-irrational number $x$ and all positive $n$ and $k$,
one has
$$
\min \{ \Theta_{n-1},\Theta_n,\ldots,\Theta_{n+k(p-1)}\} <
 c_{k},
$$
with $$
c_{k}=\frac{-\tau_{k-1}}{1+(\lambda-1)\tau_{k-1}} \quad \textrm{and} \quad  \tau_{k}=\left[\left(-1:2,\left(-1:1\right)^{p-2}\right)^k, (-2,3)\right] .  
$$
The constant $c_{k}$ is best possible.
\end{Theorem}
%

A similar theorem was derived for the case that $q$ is odd. In this article we derive Borel results for both even and odd $\alpha$-Rosen fractions. 

\subsection{\texorpdfstring{$\alpha$}{alpha}-Rosen fractions}
\label{sec: rosen intro}

Let $q\geq 3, \lambda=\lambda_q=2\cos\frac{\pi}{q}$ and let $\alpha \in \left[\frac12,\frac{1}{\lambda}\right]$. The $\alpha$-Rosen fraction operator $T_\alpha:[(\alpha-1)\lambda,\alpha\lambda) \rightarrow [(\alpha-1)\lambda,\alpha\lambda)$ is defined by
\begin{equation}
\label{def: chap 3 Talpha}
T_{\alpha}(x) =  \frac{\varepsilon}{x}- \lambda \left \lfloor\, \frac{\varepsilon}{\lambda x}  + 1 -\alpha \right\rfloor \textrm{ if } x\neq 0 \textrm{ and } T_{\alpha}(0):=0.
\end{equation}

Repeatedly applying this operator to $x \in [(\alpha-1)\lambda,\alpha\lambda)$ yields the $\alpha$-Rosen expansion of $x$. Put 
\begin{equation}
\label{eq: alpha digits}
d(x) = \left\lfloor \left| \frac{1}{\lambda x} \right| + 1 -\alpha \right\rfloor \ \quad \textrm{and} \quad 
\varepsilon(x)=\rm{sgn}(x).
\end{equation} 

Furthermore, for $n \geq 1$ with $T_{\alpha}^{n-1}(x) \neq 0$ put
\[
\varepsilon_{n}(x)=\varepsilon_n = \epsilon(T^{n-1}_\alpha (x)) \textrm{ and } d_n(x)=d_n=d(T^{n-1}_\alpha (x)).
\]
This yields a continued fraction of the type
\[
x = \displaystyle{\frac{\varepsilon_1}{d_1\lambda +  \displaystyle{\frac{\varepsilon_2}{d_2\lambda + \dots}}}} = [\varepsilon_1:d_1,\varepsilon_2:d_2,\dots],
\]
where $\varepsilon \in \{ \pm 1\}$ and $d_i \in \N^+$. 

In this article we derive a Borel-type result for $\alpha$-Rosen fractions. Let $q$ be fixed. Haas and Series~\cite{HS} showed that for every $G_q$-irrational $x$ there exist infinitely many $G_q$-rationals
$r/s$, such that $s^2\left|x-\frac{r}{s}\right| \leq \mathcal{H}_q$, where the Hurwitz constant $\mathcal{H}_q$ is given by
\begin{equation}
\label{eq: alpha hq}
\mathcal{H}_q=\begin{cases}
\displaystyle{\frac{1}{2}} & \quad \textrm{if $q$ is even,}\\
\displaystyle{\frac{1}{\sqrt{\lambda^2-4\lambda+8}}} & \quad \textrm{if $q$ is odd.}\\
\end{cases}
\end{equation}

In this paper the following Borel result is obtained.
\begin{Theorem}\label{thm: article Borel alpha}
Let $\alpha \in [1/2,1/\lambda]$ and denote the $n$th $\alpha$-Rosen convergent by $p_n/q_n$. For every $G_q$-irrational $x$ there are infinitely many $n \in
\N$ for which
\[ 
q_n^2 \left| x - \frac{p_n}{q_n} \right| \leq \mathcal{H}_q 
\]
The constant $\mathcal{H}_q$ is best possible.
\end{Theorem}

We remarked that for regular continued fractions the results of Borel and Legendre imply Hurwitz's result. For Rosen fractions, the case $\alpha=\frac12$ it follows from Nakada~\cite{N2} that the Legendre constant is smaller than the Hurwitz constant $\mathcal{H}_q$ (both in the odd and even case). This means that there might exist $G_q$-rationals with ``quality'' smaller than $\mathcal{H}_q$ that are not found as Rosen-convergents. So a direct continued fraction proof of the generalization by Haas and Series~\cite{HS} of Hurwitz's results can not be given for standard Rosen fractions. 

\subsection{Legendre and Lenstra constants}
\label{sec: intro lenstra hurwitz}

In the early 1980s H.W. Lenstra conjectured that for regular continued fractions for almost all $x$ and all $z \in [0,1]$, the limit
\[
\lim_{n \rightarrow \infty} \frac{1}{n} \# \{ 1 \leq j \leq n | \Theta_j(x) \leq z \}
\]
exists and equals the distribution function $F$ defined by
$$
F(z) = \begin{cases}
\displaystyle{\frac{z}{\log 2}} \quad& \textrm{if } 0\leq z \leq \frac12\\
\\
\displaystyle{\frac{1-z+\log 2z}{\log 2}} \quad  &\textrm{if } \frac12 \leq z \leq 1.
\end{cases}
$$
A version of this conjecture had been formulated by W. Doeblin~\cite{Doeb} before. In 1983 W. Bosma \emph{et al.}~\cite{BJW} proved the Doeblin-Lenstra-conjecture for regular continued fractions and Nakada's $\alpha$-expansions for $\alpha \in \left[\frac12,1\right]$. 

A prominent feature of $F$ is that there exists a unique positive constant $\mathcal{L}$ such that $F(z)$ is linear for $z \in [0,\mathcal{L}]$.  For the RCF we have $\mathcal{L}=\frac12$. In~\cite{N2}, Nakada calls $\mathcal{L}$ the Lenstra constant and shows that for a large class of continued fractions this Lenstra constant is equal to the Legendre constant. Recently it was shown in~\cite{KNS} that the so-called mediant map has a Legendre constant larger than the Hurwitz constant $\mathcal{H}_q$, thus yielding a Hurwitz result. These results were obtained using the Lenstra constant. We derive a Hurwitz result for some $\alpha$-Rosen fractions.

The outline of this article is as follows. In Section~\ref{sec:geo} we give some general definitions for the natural extensions for $\alpha$-Rosen fractions and explain briefly how our method works. The even and odd case have different properties and we handle the details in two separate sections. The Borelt result for the different subcases of even $\alpha$-Rosen fractions are derived in Section~\ref{sec:even}, and the odd case is given in Section~\ref{sec:odd}. In Section~\ref{sec: Hurwitz rosen} we find the Lenstra constants $\mathcal{L}_\alpha$ for $\alpha$-Rosen fractions  and conclude for which values of $\alpha$ we can derive a Hurwitz result from this. 

\section{The natural extension for \texorpdfstring{$\alpha$}{alpha}-Rosen fractions}
\label{sec:geo}

In this section we introduce the necessary notation. Recall from~(\ref{def: chap 3 Talpha}) that 
$$
\displaystyle{T_\alpha(t) = \frac{\varepsilon(t)}{t} -d(t)\lambda}  \quad \textrm{with} \quad \varepsilon(t) = \textrm{sgn}(t)  \quad \textrm{and} \quad d(t) = \displaystyle{\left\lfloor \frac{\varepsilon}{\lambda t}+ 1-\alpha \right\rfloor}.
$$

\begin{Definition} \label{natural extension} For fixed $q$ and $\alpha$ the natural extension map $\mathcal{T}_\alpha :\Omega_\alpha \rightarrow \Omega_\alpha$ is given by
\label{def: T(x,y)_alpha}
\[ 
\mathcal{T_\alpha}(t,v) = \left( T_\alpha(t), \frac{1}{d(t)\lambda + \varepsilon(t) v} \right).
\]

\end{Definition}
The shape of the domain $\Omega_\alpha$ on which the two-dimensional map $\mathcal{T}_\alpha$ is bijective a.e.  was constructed in~\cite{DKS}.  We derive our results using a geometric method based on the natural extensions $\Omega_\alpha$. The shape of $\Omega_\alpha$ depends on $\alpha$ and we give the explicit formulas for each of the different cases in the appropriate sections; see e.g.~the beginning of Subsection~\ref{sec: even first} for $\Omega_\alpha$ when $q$ is even and $\alpha \in \left( \frac{1}{2},\frac{1}{\lambda}\right)$. The natural extension also depends on $q$, but for ease of notation we suppress this dependence and write $\Omega_\alpha$ in stead of $\Omega_{\alpha,q}$.

We use constants $l_n$ and $r_n$ to describe $\Omega_\alpha$, where
$$
\begin{aligned}
l_0&=(\alpha-1)\lambda &\textrm{ and } &l_n=T^n_\alpha(l_0), &\textrm{ for }n\geq0,\\
r_0&=\alpha \lambda &\textrm{ and } &r_n=T^n_\alpha(r_0), &\textrm{ for }n\geq0.
\end{aligned}
$$
The orbit of $-\frac{\lambda}{2}$ in the case $\alpha=\frac12$ plays an important role in describing the natural extensions. We define $\varphi_j = T_\frac{1}{2}^j\left(-\frac{\lambda}{2}\right)$.

We set $\delta_d = \frac{1}{(\alpha+d)\lambda}$ for all $d \geq 1$. So if $\delta_{d} < x \leq \delta_{d-1}$, we have $d(x) = d$ and $\varepsilon(x)=+1$. For $x$ with $-\delta_{d-1} \leq x < -\delta_d$ we have $d(x)=d$ and $\varepsilon(x)=-1$; also see~(\ref{eq: alpha digits}).

We often use the auxiliary sequence $B_n$ given by
\begin{equation}
\label{eq: bn chap 3}
B_0 =0, \quad B_1=1, \quad B_n=\lambda B_{n-1}-B_{n-2}, \quad \textrm{ for } n=2,3,\dots.
\end{equation}

Note that $B_n = \sin\frac{n \pi}{q}/\sin \frac{\pi}{q}$. If $q=2p$ for $p \geq 2$, we find from $\sin \frac{(p-1)\pi}{2p}=\sin \frac{(p+1)\pi}{2p}$ that
\begin{equation}
\label{eq: B_n even relations}
B_{p-1} = B_{p+1} = \frac{\lambda}{2}B_p \quad \textrm{and} \quad B_{p-2} = \left( \frac{\lambda^2}{2}-1 \right)B_p.
\end{equation}
Similarly in the odd case with $q=2h+3$ for $h \in \N$ we have that
\begin{equation}
\label{eq: B_n odd relations}
B_{h+1} = B_{h+2}, \quad B_h = (\lambda-1)B_{h+1} \quad \textrm{and} \quad B_{h-1} = \left(\lambda^2-\lambda-1 \right)B_{h+1}.
\end{equation}

We define for $x\in [l_0,r_0)$ with $\alpha$-Rosen expansion $[\varepsilon_1:d_1,\varepsilon_2:d_2,\dots]$  the \emph{future} $t_n$ and the \emph{past} $v_n$ of $x$ at time $n \geq 1$ by
$$
t_n =
[\varepsilon_{n+1}:d_{n+1},\varepsilon_{n+2}:d_{n+2},\ldots] \quad \textrm{and} \quad
v_n = [1:d_n,\varepsilon_{n}:d_{n-1},\ldots,\varepsilon_2:d_1].
$$ 
We set $t_0 = x$ and $v_0=0$.

 \begin{Rmk}
Note that $\mathcal{T_\alpha}^n(x,0) = (t_n,v_n)$ for $n\geq 0$. 
\end{Rmk} 

 The $(n-1)$st and $n$th approximation coefficients of $x$ can be given in terms of $t_n$ and $v_n$ (see Section 5.1.2 of~\cite{DK}) as
\begin{equation} \label{eq: formulae for theta}
\Theta_{n-1}=\Theta_{n-1}(t_n,v_n)=\frac{v_n}{1+t_nv_n} \quad \textrm{ and } \quad \Theta_{n}=\Theta_{n}(t_n,v_n)=\frac{\varepsilon_{n+1}t_n}{1+t_nv_n}.
\end{equation}
Often it is convenient to use $\Theta_{m}(t_{n+1},v_{n+1})=\Theta_{m+1}(t_n,v_n)$.

\begin{Lemma} 
The $(n+1)$st approximation coefficient of $x$ can be expressed in terms of $d_{n+1},\varepsilon_{n+1}$ and $\varepsilon_{n+2}$ by
\begin{equation}
\label{eq: formula Theta_n+1}
\Theta_{n+1}=\Theta_{n+1}(t_n,v_n)=\frac{\varepsilon_{n+2}(1-\varepsilon_{n+1}d_{n+1}t_n\lambda)(\lambda d_{n+1}
+\varepsilon_{n+1}v_n)}{1+t_nv_n}.
\end{equation}
\end{Lemma}
\begin{proof}
First we use~(\ref{eq: formulae for theta}) to write
$$
\Theta_{n+1} =\Theta_n (t_{n+1},v_{n+1}) \frac{\varepsilon_{n+2} \, t_{n+1}}{1+t_{n+1}\,v_{n+1}} = \varepsilon_{n+2}\,t_{n+1}\,\frac{\Theta_n}{v_{n+1}}
= \frac{\varepsilon_{n+2}\,\varepsilon_{n+1}\,\frac{t_{n+1}t_n}{v_{n+1}}}{1+t_nv_n}.
$$
Then we use $t_{n+1}=\frac{\varepsilon_{n+1}}{t_n}-d_{n+1}\lambda$ and $v_{n+1} = \frac{1}{\lambda d_{n+1}+\varepsilon_{n+1}v_n}$ to find~(\ref{eq: formula Theta_n+1}).
\end{proof}

In view of~(\ref{eq: formulae for theta}) we define functions $f$ and $g$  on $[l_0,r_0)$ by 
\begin{equation} \label{eq: def f and g}
f(x)=\frac{\mathcal{H}_q}{1-\mathcal{H}_qx} \quad\textrm{and } \quad g(x)=\frac{|x|-\mathcal{H}_q}{\mathcal{H}_qx}.
\end{equation}
Then for points $(t_n,v_n) \in \Omega_\alpha$  one has
\begin{equation}
\label{eq: relations f g theta}
\Theta_{n-1} \leq  \mathcal{H}_q \Leftrightarrow v_n \leq f(t_n) \quad \textrm{ and } \quad
\Theta_{n} \leq \mathcal{H}_q \Leftrightarrow
\begin{cases}
v_n \leq g(t_n)  \quad \textrm{if } t_n < 0\\
\\
v_n \geq g(t_n) 
\quad \textrm{if } t_n \geq 0.
\end{cases}
\end{equation}

We define $\mathcal{D}$ as
\begin{equation}
\label{def:D}
\mathcal{D}=\left\{(t,v) \in \Omega_\alpha | \min \left\{ \frac{v}{1+tv},\frac{|t|}{1+tv}\right\} > \mathcal{H}_q \right\},
\end{equation}
so $\min \{ \Theta_{n-1},\Theta_n \} > \mathcal{H}_q$ if and only if $(t_n,v_n) \in \mathcal{ D}$.

See Figure~\ref{fig:q=4(i)} for an example of the position of $\mathcal{D}$ and of the graphs of $f$ and $g$ in~$\Omega_\alpha$ for $q=4$.


\section{Tong's spectrum for even \texorpdfstring{$\alpha$}{alpha}-Rosen fractions}
\label{sec:even}
Let $q=2p$ for $p \in \N^+$,  $p \geq 2$ and set $\lambda=2\cos\frac{\pi}{q}$. As shown in~\cite{DKS} there are three subcases for the shape of $\Omega_\alpha$: we need to study $\alpha=\frac12, \alpha \in \left(\frac{1}{2},\frac{1}{\lambda} \right)$ and $\alpha=\frac{1}{\lambda}$ separately. The following theorem was essential in the construction of the natural extensions and gives the ordering of the $l_n$ and $r_n$.

\begin{Theorem}
\label{th: position l r even}\cite{DKS}
Let $q=2p, p\in \N, p\geq 2$ and let $l_n$ and $r_n$ be defined as before. If $\frac{1}{2} <\alpha <  \frac{1}{\lambda}$, then we have that
$$
-1 <l_0 < r_1 <l_1 <\ldots < r_{p-2} <l_{p-2} <-\delta_1 <r_{p-1} <0 <l_{p-1} <r_0 <1,
$$
$d_p(r_0) = d_p(l_0)+1$ and $l_p = r_p$. If $\alpha=\frac12$, then we have that
$$
-1 < l_0<r_1=l_1 < \ldots <r_{p-2} = l_{p-2} <-\delta_1 <r_{p-1}=0 = l_{p-1}  <r_0<1.
$$
If $\alpha=\frac{1}{\lambda}$, then we have that
$$
-1< l_0=r_1 < l_1=r_2 < \ldots <l_{p-2}  =-\delta_1= r_{p-1} <0 < r_0=1.
$$
\end{Theorem}

Let $k \geq 1 $ be an integer and put
\begin{equation}
\label{def hat t and v even}
(\tau_{k},\nu_k) = \begin{cases}
\mathcal{T}_\alpha^{-k(p-1)} \left( \frac{-2}{3\lambda} ,\lambda-1\right)  &\textrm{ if } \alpha = \frac{1}{\lambda};\\
\\
\mathcal{T}_\alpha^{-k(p-1)} \left(  -\delta_1 ,\lambda-1\right)   &\textrm{ otherwise.}
\end{cases}
\end{equation}

We prove the following result in this section.

\begin{Theorem}
\label{th: tong even general} Fix an even $q = 2p$ with $p \ge 3$ and let $\alpha \in \left[\frac12,\frac{1}{\lambda}\right]$.  There exists a positive integer $K$ such that for
every $G_q$-irrational number $x$ and all positive $n$ and $k >K$, 
one has
$$
\min \{ \Theta_{n-1},\Theta_n,\ldots,\Theta_{n+k(p-1)}\} < c_{k} \quad \textrm{with } c_{k}=\frac{-\tau_{k-1}}{1+\tau_{k-1}\nu_{k-1}}.  
$$
For every integer $k \geq 1$ we have $c_{k+1} < c_k $. Furthermore $\displaystyle{\lim_{k \rightarrow \infty} c_{k}} = \frac12$.
\end{Theorem}

The case $\alpha = \frac12$ was proven in~\cite{KSS}. The proof for $\alpha \in \left(\frac12,\frac{1}{\lambda}\right)$ is given in Section~\ref{sec: even first}. In Section~\ref{sec: even lambda} we use that the natural extension for the case $\alpha=\frac{1}{\lambda}$ is the mirror image of the one for $\alpha=\frac12$ to prove Theorem~\ref{th: tong even general} for $\alpha=\frac{1}{\lambda}$.

\subsection{Even case with \texorpdfstring{$\alpha \in  (\frac12, \frac{1}{\lambda})$}{alpha in the interval}}
\label{sec: even first}
In this section we assume that $\alpha \in  (\frac12, \frac{1}{\lambda})$. In~\cite{DKS} the shape of $\Omega_\alpha$was determined.  

\begin{Definition}
Set 
$$
\begin{aligned}
J_{2n-1} &= [l_{n-1},r_n)  \quad \textrm{and} \quad J_{2n}=[r_n,l_n) \quad \textrm{for}\quad  n=1,2,\dots,p-1, \\
J_{2p-1}&=[l_{p-1},r_0)  \quad \textrm{and} \quad \\
H_1 &=\, \frac{1}{\lambda+1}, \quad  H_2 \, =\, \frac{1}{\lambda}  \quad \textrm{and} \quad H_n  =\, \frac{1}{\lambda-H_{n-2}} \quad  \text{for} \quad  n = 3,4,\dots ,2p-1.\\
\end{aligned}
$$

The shape of $\Omega_\alpha$ upon which $\mathcal{T}_\alpha$ is bijective a.e.~is given by
 $$
   \Omega_\alpha = \bigcup_{n=1}^{2p-1} J_n \times [0,H_n].
$$
\end{Definition}

We define
$$
\Omega_\alpha^+ = \left\{ (t,v) \in \Omega_\alpha \, |\, t > 0 \right\}.
$$
 
 
From the above description of the natural extension $\Omega_{\alpha}$ it follows that the natural extension has $2p-1$ heights $H_1,\dots,H_{2p-1}$. In~\cite{DKS} is is shown that $H_{i+1} > H_i$ for $i=1,\dots, 2p -2$,
$$H_{2p-3}=\lambda-1,\quad H_{2p-2}=\frac{\lambda}{2} , \quad H_{2p-1}=1,$$ 
and
\begin{equation}
\label{eq: lp1 rp1 middle even}
l_{p-2} = \frac{ \alpha\lambda^2-2}{(-\alpha\lambda^2+2\alpha+1)\lambda} ,
l_{p-1}=\frac{(2\alpha-1)\lambda}{2-\alpha\lambda^2}  \textrm{ and } r_{p-1}=-\frac{(2\alpha-1)\lambda}{2-(1-\alpha)\lambda^2}.
\end{equation}

Note that Theorem~\ref{th: tong even general} gives a result for  $q\geq 6$. The case $q=4$ behaves slightly differently, which we show in the following subsection.

\subsubsection{The case \texorpdfstring{$q = 4$}{q = 4}} 
In this subsection we assume that $q=2p=4$, so $\lambda = \sqrt{2}$. In this case we have $H_1 = \sqrt{2}-1, H_2=\frac{1}{2}\sqrt{2}$ and $H_3=1$. We prove the following result.

\begin{Theorem}
\label{th: tong even q4} Let $\lambda=2\cos\frac{\pi}{4}=\sqrt{2}$, let $\alpha \in \left(\frac12,\frac{1}{\lambda}\right)$ and let $\tau_k$ be as given in~(\ref{def hat t and v even}).  There exists a positive integer $K$ such that for
every $G_q$-irrational number $x$ and all positive $n$ and $k >K$,
one has
$$
\min \{ \Theta_{n-1},\Theta_n,\ldots,\Theta_{n+k}\} < c_{k} \quad \textrm{with } c_{k}=\frac{\sqrt{2}-1}{1+\tau_{k-1}(\sqrt{2}-1)}.  
$$
For every integer $k \geq 1$ we have $c_{k+1} < c_k $. Furthermore $\displaystyle{\lim_{k \rightarrow \infty} c_{k}} = \frac12$.
\end{Theorem}

We start by determining the shape of the region ${\mathcal D}\subset \Omega_{\alpha}$, where
$\min \{ \Theta_{n-1},\Theta_n\} > \frac{1}{2}$ as defined in~(\ref{def:D}).  

\begin{Lemma}\label{lem:q=4} Put $\alpha_0:=\frac{4+\sqrt{2}}{8} = 0.676\dots$. For $\alpha\in \left( \frac{1}{2},\alpha_0\right)$ the region ${\mathcal D}$ consists of one
 component $\mathcal{ D}_1$, which is bounded by the lines $t=l_0$, $v=H_1$, and the graph of $f$; see Figure~\ref{fig:q=4(i)}.
If $\alpha\in \left[ \alpha_0,\, \frac{\sqrt{2}}{2}\right)$, then ${\mathcal D}$ consists of two  components: ${\mathcal D}_1$ and ${\mathcal D}_2$, where $\mathcal{ D}_2$
is the region bounded by the lines $t=r_1$, $v=H_2$, and the graph of $g$; see Figure~\ref{fig:q=4(ii)}.
\end{Lemma}
\begin{proof}
Recall that ${\mathcal H}_q=\frac{1}{2}$. First assume that $t\geq 0$. The graphs of $f$ and $g$ do not intersect for $t\leq r_0$. Thus every point $(t_n,v_n)\in \Omega_\alpha^+$  is below the graph of $f$ or above the graph of $g$. By~(\ref{eq: relations f g theta}) we have that $\min \{ \Theta_{n-1},\Theta_n\} <\frac{1}{2}$.

Assume that $t<0$. The graphs of $f$ and $g$ intersect with the line $v=H_1$ in the point $(1-\sqrt{2}, \sqrt{2}-1)$, and we find that $l_0 <-\frac{1}{2}< -\delta_1<1-\sqrt{2}<r_1$ for $\alpha \in (\frac{1}{2},\frac{1}{\lambda})$. Since both $f$ and $g$ are monotonically increasing we find that for $l_0 <t<r_1$ the intersection of $\mathcal{ D}$ and $\Omega_\alpha$ is ${\mathcal D}_1$.

\begin{figure}[!ht]
\includegraphics[height=50mm]{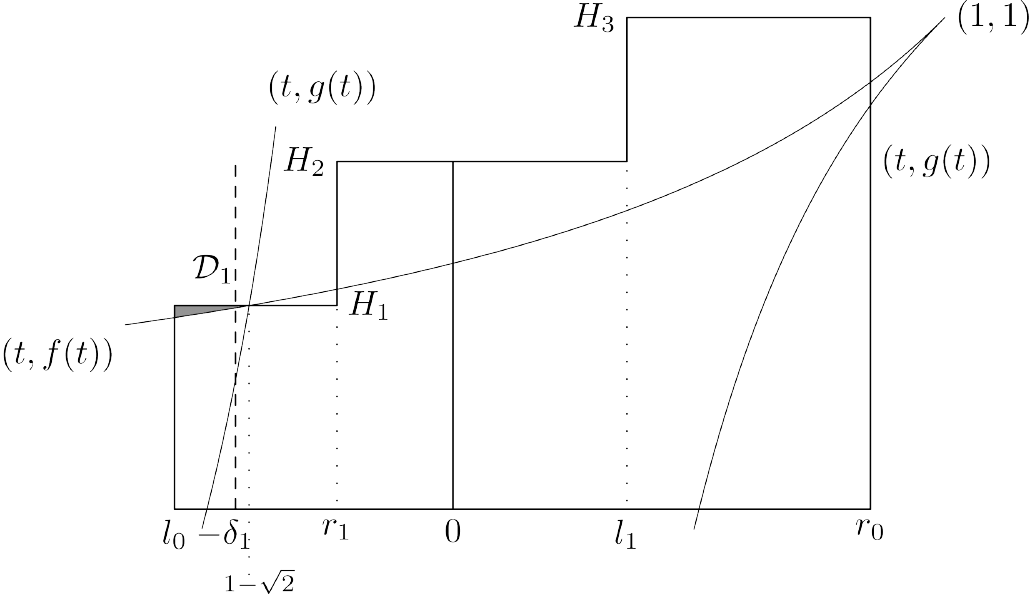}
\caption{In this picture $q=4$ and  $\alpha =0.6 < \alpha_0 = 0.676\dots $. The region ${\mathcal D}$ consists of one  component $\mathcal{ D}_1$.}
\label{fig:q=4(i)}
\end{figure}

One easily checks that $g(r_1)\leq H_2$ if and only if $\alpha \geq \alpha_0=\frac{4+\sqrt{2}}{8}=0.67677\dots$. So we find that for $\alpha > \alpha_0$ and $r_1<t<0$ the intersection of $\mathcal{ D}$ and $\Omega_\alpha$ is ${\mathcal D}_2$ (for $\alpha = \alpha_0$ the region ${\mathcal D}_2$ consists of exactly one point, $(r_1,g(r_1))$); see Figure~\ref{fig:q=4(ii)}.
\end{proof}

\begin{figure}[!ht]
\includegraphics[height=60mm]{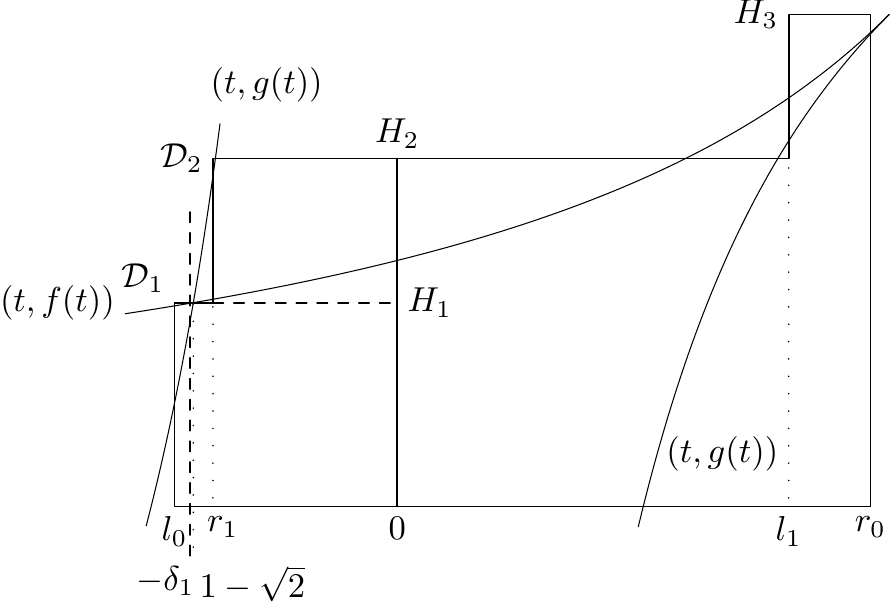}
\caption{In this picture  $q=4$ and $\alpha = 0.68 >  \alpha_0= 0.676\dots $. The region ${\mathcal D}$ consists of two  components.}
\label{fig:q=4(ii)}
\end{figure}
%

\emph{Proof of Theorem~\ref{th: tong even q4}.} Recall that $(\tau_{k},\nu_k) = \mathcal{T}_\alpha^{-k} \left(  -\delta_1 ,\lambda-1\right)$. For $q=4$ we find $\tau_1 = \frac{-1}{2\sqrt{2}-\delta_1}$, $\tau_k=[(-1:2)^{k},-\delta_1]$ and $\nu_k=\sqrt{2}-1$ for all $k$. Since $T_\alpha$ is strictly increasing on the interval $[-\delta_1,-\delta_2)$ we find $\tau_{k-1} <\tau_k$. In this case $\displaystyle{\lim_{k\rightarrow \infty} \tau_k = [\overline{(-1:2)}]}= \frac{-1}{\sqrt{2}+1}=1-\sqrt{2}$. We conclude that $\displaystyle{\lim_{k\rightarrow \infty} \tau_k =1-\sqrt{2}}.$ 

We find $c_{k+1} < c_k $, $\displaystyle{\lim_{k\rightarrow \infty} c_k} = \frac{\sqrt{2}-1}{1+(1-\sqrt{2})(\sqrt{2}-1)}= \frac12$ and conclude that $c_k >\frac12$ for all $k$.

We now focus on the orbit of points in $\mathcal{ D}$ and start with $\mathcal{ D}_1$. First note that ${\mathcal T}_{\alpha}\left( [l_0,-\delta_1)\times [0,H_1]\right) = [l_1, r_0)\times [H_2,1]\subset \{ (t,v)\in \Omega_{\alpha} |\, t\geq 0\}$. Therefore, if $(t_n,v_n)\in {\mathcal D}_1$ and $t_n\leq -\delta_1$, we have that $\min \{ \Theta_{n-1},\Theta_n\} >\frac{1}{2}$, while $\Theta_{n+1}<\frac{1}{2}$. For these points we have proven Theorem~\ref{th: tong even q4} with $K=1$.

Note that $(1-\sqrt{2},\sqrt{2}-1)$ is a fixed-point of ${\mathcal T}_{\alpha}$. In particular, we have that  $(1-\sqrt{2},\sqrt{2}-1)$ is a repellent fixed-point for the first-coordinate map of ${\mathcal T}_{\alpha}$, and an attractive fixed point for the second coordinate map of ${\mathcal T}_{\alpha}$. Thus points $(t,v) \in \mathcal{ D}_1$ with $t > -\delta_1$ move ``left and up'' under $\mathcal{T}_\alpha$. Noting that the graph of $f$ is strictly increasing, we have that the region $\{ (t,v)\in {\mathcal D}_1|\, t \geq -\delta_1\}$ is mapped by ${\mathcal T}_{\alpha}$ inside ${\mathcal D}_1$.

Setting ${\mathcal D}_{1,k}:=\{ (t,v)\in {\mathcal D}_1|\, \tau_{k-1}\leq t <\tau_k\}$, for $k\geq 1$, and ${\mathcal D}_{1,0}:=\{ (t,v)\in {\mathcal D}_1|\, t<-\delta_1\}$, by definition of $\tau_k$ and ${\mathcal D}_{1,k}$, we have for $k\geq 1$ that 
$$
(t,v)\in {\mathcal D}_{1,k} \quad \textrm{implies} \quad {\mathcal T}_{\alpha}(t,v)\in {\mathcal D}_{1,k-1}.
$$

We determine the maximum of $\Theta_{n-1},\Theta_n$ and $\Theta_{n+1}$ on $\mathcal{ D}_{1,k}$ for $k \geq 1$.

\begin{Lemma}
\label{lem: theta on d q4}
Let $k\geq 1$ and $(t_n,v_n) \in  \mathcal{ D}_{1,k}$. Then
\begin{equation}
\label{eq: theta order q4}
\Theta_{n-1} \leq \Theta_n \leq \Theta_{n+1}. 
\end{equation}
\end{Lemma}
\begin{proof}
On $\mathcal{ D}_{1,k}$ we have $d_{n+1}=2$ and $\varepsilon_{n+1}= \varepsilon_{n+2}=-1$. From~(\ref{eq: formulae for theta}) we find
$$
\Theta_{n-1} \leq \Theta_{n} \quad \textrm{if and only if} \quad  v_n \leq -t_n, 
$$
and the latter inequality is true in view of the fixed point. From~(\ref{eq: formula Theta_n+1}) we have $\Theta_{n+1} = \frac{-(1+2\sqrt{2}t_n)(2\sqrt{2}-v_n)}{1+t_nv_n}$ and we find
$$
\Theta_{n} \leq \Theta_{n+1}  \quad \textrm{if and only if} \quad  v_n \leq \frac{7t_n +2\sqrt{2}}{1+2\sqrt{2}t_n}.
$$
On $\mathcal{ D}_{1,k}$ the function $v(t) = \frac{7t +2\sqrt{2}}{1+2\sqrt{2}t}$ is decreasing in $t$ and $v(1-\sqrt{2})=\sqrt{2}-1$. So $\Theta_{n} \leq \Theta_{n+1}$ on $\mathcal{ D}_{1,k}$ for $k\geq 1$. 
\end{proof}

We conclude that for every point $(t_n,v_n) \in \mathcal{ D}_{1,k}$ for $k \geq 1$
$$
\min \{ \Theta_{n-1},\Theta_n,\Theta_{n+1} \} \leq \max_{(t,v) \in \mathcal{ D}_{1,k} } \Theta_{n-1} (t,v).
$$

We determine the maximum of $\Theta_{n-1}$ on $\mathcal{ D}_{1,k}$. The partial derivatives of $\Theta_{n-1}$ are given by 
\[
\frac{\partial \Theta_{n-1}}{\partial t_n} = \frac{-v_n^2}{(1+t_nv_n)^2} < 0 \quad \textrm{ and } \quad \frac{\partial \Theta_{n-1}}{\partial v_n} = \frac{1}{(1+t_nv_n)^2}>0,
\]
so on $\mathcal{ D}_{1,1}$ we see that $\Theta_{n-1}$ attains its maximum in $(-\delta_1,\sqrt{2}-1)$. Similarly on $\mathcal{ D}_{1,k}$ we find that $\Theta_{n-1}$ attains its maximum in $(\tau_k,\sqrt{2}-1)$, which is the top left-hand vertex of $\mathcal{ D}_{1,k}$. We find 
$$
c_1=\frac{\sqrt{2}-1}{1-\delta_1(\sqrt{2}-1)}  \quad \textrm{and} \quad c_k=\frac{\sqrt{2}-1}{1+\tau_{k-1}(\sqrt{2}-1)}.
$$

One sees that if $(t_n,v_n) \in \mathcal{ D}_{1,k}$ for some $k\geq 3$, then $(t_{n+1},v_{n+1}) \in \mathcal{ D}_{1,k-1}$, $(t_{n+2},v_{n+2}) \in \mathcal{ D}_{1,k-2}, \ldots,(t_{n+k-1},v_{n+k-1}) \in \mathcal{ D}_{1,1} $.  It follows that
\begin{eqnarray}
\nonumber
(t_n,v_n)\in {\mathcal D}_{1,k} &  \textrm{implies}  & \frac{1}{2} < \min
\{ \Theta_{n-1},\dots ,\Theta_{n+k}\} <
\Theta_{n-1}(\tau_{k-1},\sqrt{2}-1)=c_k \\
& & \mbox{ and } \Theta_{n+k+1}<\frac{1}{2}.\label{implication}
\end{eqnarray}

For $\alpha <\alpha_0$ the above implication~(\ref{implication}) is actually an equivalence, since ${\mathcal D}_2$ is void for these values of $\alpha$. Thus Theorem~\ref{th: tong even q4} for the case $q=4$ and $\alpha <\alpha_0$ follows with $K=1$.

We continue by studying the orbit of points in $\mathcal{D}_2$ and assume that $\alpha \in \left[\alpha_0,\frac{1}{\lambda}\right)$, so $\mathcal{D}_2$ is non-empty. 

It follows from the fact that $(1-\sqrt{2}, \sqrt{2}-1)$ is a repellent fixed-point on the first coordinate map of ${\mathcal T}_{\alpha}$, and an attractive fixed-point on the second coordinate map of ${\mathcal T}_{\alpha}$,
that for $(t,v)\in {\mathcal D}_2$
$$
1-\sqrt{2}<r_1<t<\pi_1({\mathcal T}_{\alpha}(t,v))\quad
{\mbox{and}}\quad H_2>v>\pi_2({\mathcal T}_{\alpha}(t,v)) >
\sqrt{2}-1,
$$
(here $\pi_i$ is the projection on the $i$th coordinate), i.e., ${\mathcal T}_{\alpha}$ ``moves'' the point $(t,v)\in {\mathcal
D}_2$ to the right, and ``downwards towards'' $\sqrt{2}-1$. 

Let $(t_1,\sqrt{2}-1+v_1)$ be a point in $\mathcal{ D}_2$. The lowest point in $\overline{\mathcal{ D}_2}$, the closure of $\mathcal{ D}_2$, is given by $(r_1,g(r_1))=\left(\frac{1-2\alpha}{\sqrt{2}\alpha},\frac{(\sqrt{2}-4)\alpha+2}{2\alpha-1}\right)$, so $0< \frac{-(2+\sqrt{2})\alpha + 1+\sqrt{2}}{2\alpha-1} < v_1 $ for every point $(t_1,\sqrt{2}-1+v_1) \in \mathcal{ D}_2$ and trivially $v_1<1$.

For the second coordinate we find
$$
\pi_2(\mathcal{T}_\alpha(t_1,\sqrt{2}-1+v_1)) = \frac{1}{\sqrt{2}+1-v_1} = \sqrt{2}-1 + \frac{\sqrt{2}-1}{\sqrt{2}+1-v_1} v_1.
$$

For all points in $\mathcal{ D}_2$ we have $d = 2$ and $\varepsilon=-1$. Thus in every consecutive step the second coordinate will be a factor $\frac{\sqrt{2}-1}{\sqrt{2}+1-v_1} < \frac{\sqrt{2}-1}{\sqrt{2}}<1$ closer to the value $\sqrt{2}-1$. Hence there exists a smallest positive integer $K$ such that for all $(t,v) \in\mathcal{ D}_2$ one has ${\mathcal T}_{\alpha}^{K}(t,v)\not\in {\mathcal D}_2$. In words: the region ${\mathcal D}_2$ is ``flushed'' out of $\mathcal{ D}$ in $K$ steps, and the implication in~(\ref{implication}) is an equivalence for $k>K$. This proves Theorem~\ref{th: tong even q4}.
\hfill $\Box$\medskip\

\begin{Rmk} More can be said with (considerable) effort. We start by deriving $\alpha_1$ such that for $\alpha\in
[\alpha_0,\alpha_1]$ the region $\mathcal{D}_2$ is non-empty, but flushed after one iteration of ${\mathcal T}_{\alpha}$.
We find $\alpha_1$ by solving for which value of $\alpha$ we have that the point $\mathcal{T}_\alpha(r_1,H_2)=\left(r_2,\frac{\sqrt{2}}{3}\right)$is on the graph of $g$. Using that $d(r_1) = 2$ for $\alpha \in \left[\alpha_0,\frac{1}{\lambda}\right)$ we find that the only solution is given by $\alpha_1:= \frac{24+\sqrt{2}}{36}=0.70595\dots$.

So if $\alpha \in [\alpha_0,\alpha_1] = \left[ \frac{4+\sqrt{2}}{8}, \frac{24+\sqrt{2}}{36}\right] $ and $(t_n,v_n)\in {\mathcal D}_2$, then $\min \{ \Theta_{n-1},\Theta_n\} >\frac{1}{2}$, and $\Theta_{n+1}<\frac{1}{2}$.

For $\alpha > \alpha_1$ and $i\geq 1$, we define the pre-images $g_i$ of $g$ for $t\in [1-\sqrt{2},\beta ]$ by
$$
v=g_i(t)\quad \Leftrightarrow \quad \pi_2({\mathcal
T}_{\alpha}^i(t,v)) = g\big( \pi_1({\mathcal T}_{\alpha}^i(t,v))
\big) ,
$$
i.e., the point $(t,v)$ is on the graph of $g_i$ if and only if ${\mathcal T}_{\alpha}^i(t,v)$ is on the graph of $g$. Note that for every $i\geq 1$ one has that $(1-\sqrt{2},\sqrt{2}-1)$ is on the graph of $g_i$.

By definition of $K_{\alpha}$ it follows that the graph of $g_i$ has
a non-empty intersection with ${\mathcal D}_2$ if and only if
$i=1,\dots,K_{\alpha}-1$. These $K_{\alpha}-1$ graphs $g_i$ divide
${\mathcal D}_2$ like a ``cookie-cutter'' into regions  ${\mathcal
D}_{2,i}$ for $i=1,\dots,K_{\alpha}$; setting $g_0:=g$,
$$
{\mathcal D}_{2,i}:=\{ (t,v)\in {\mathcal D}_2|\, g_{i-1}(t)\leq v<
\min \{ H_2,g_i(t)\} \} .
$$
We have that
$$
(t_n,v_n)\in {\mathcal D}_{2,i} \quad \Rightarrow \quad \min \{
\Theta_{n-1},\dots,\Theta_{n+i-1} \} >\frac{1}{2} {\mbox{ and }}
\Theta_{n+i}<\frac{1}{2}.
$$
In principle it is possible for $(t_n,v_n)\in {\mathcal D}_2$ and $k=1,\dots,K-1$ to determine the optimal constant ${\tilde c}_k$ such that
$$
\min \{ \Theta_{n-1},\dots,\Theta_{n+k} \} < {\tilde c}_k .
$$
In this way, Theorem~\ref{th: tong even general} can be further
sharpened.

\end{Rmk}

\begin{Rmk}
From the proof of Lemma~\ref{eq: theta order q4} it easily follows that for points $(t_n,v_n) \in \mathcal{ D}_2$
$$
\min \{\Theta_{n-1},\Theta_{n},\Theta_{n+1}\} = \Theta_{n}.
$$
The partial derivatives of $\Theta_n$ on $\mathcal{ D}_2$ are given by 
\[
\frac{\partial \Theta_n}{\partial t_n} = \frac{-1}{(1+t_nv_n)^2}<0 \quad \textrm{ and } \quad \frac{\partial \Theta_n}{\partial v_n} = \frac{t_n^2}{(1+t_nv_n)^2}>0.
\]

\end{Rmk}

\begin{Ex}
An easy but tedious calculation yields that ${\mathcal T}_{\alpha}^2(r_1,H_2)$ is on the graph of $g$ if $\alpha
= \alpha_2:=\frac{140+\sqrt{2}}{200}=0.707071\dots$. 

For $\alpha \in (\alpha_1,\alpha_2]$ we have that $K_{\alpha}=2$. For $\alpha \in (\alpha_1,\alpha_2]$ the region ${\mathcal D}_2$ consists of two parts: ${\mathcal D}_{2,1}$ and ${\mathcal D}_{2,2}$. The region ${\mathcal D}_{2,1}$ is immediately flushed and is therefore not interesting for us. For $(t_n,v_n)\in {\mathcal D}_{2,2}$ we have that
$$
\frac{1}{2} < \min \{ \Theta_{n-1},\Theta_n,\Theta_{n+1}\} < \Theta_n \left( {\mathcal T}_{\alpha}(r_1,H_2)\right) = 9\sqrt{2}\alpha -6\sqrt{2}.
$$
The comparable value of $c_1$ is
$$
c_1=\frac{\delta_1}{1-\delta_1H_1}=\frac{(\alpha+1)(2-\sqrt{2})}{\sqrt{2}\alpha+2\sqrt{2} - 1}.
$$
We find that $ 9\sqrt{2}\alpha-6\sqrt{2}> \frac{(\alpha+1)(2-\sqrt{2})}{\sqrt{2}\alpha+2\sqrt{2} - 1}$ when
$ \alpha > \frac{12-9\sqrt{2}+\sqrt{378+216\sqrt{2}}}{36}=0.6944$. So we find for  $\alpha \in (\alpha_1,\alpha_2]$ that
$$
\min \{ \Theta_{n-1},\Theta_n,\Theta_{n+1}\} < 9\sqrt{2}\alpha -6\sqrt{2}.
$$

\end{Ex}


\subsubsection{Even case with \texorpdfstring{$\alpha \in  (\frac12, \frac{1}{\lambda})$}{alpha in the interval} and  \texorpdfstring{$q \geq 6 $}{q > 4}}
\label{sec: even interval}
From now on we assume $q \geq 6$. The shape of the region $\mathcal{D} \subset \Omega_\alpha$, where $\min\{\Theta_{n-1},\Theta_{n}\}>\frac12$  as defined in~(\ref{def:D}) is given in the next lemma; also see Figure~\ref{im: d in example q 6}.

\begin{Lemma}
\label{th: shape D}
For $\alpha \in (\frac12 ,\frac{\lambda^2+4\lambda-4}{2\lambda^3}]$ the region  $\mathcal{D}$ consists of two  components $\mathcal{D}_1$ and $\mathcal{D}_2$.  The subregion
$\mathcal{D}_1$ is bounded by the lines $t = l_0, v = H_1$ and the graph of $f$;  $\mathcal{D}_2$ is bounded by the graph of $g$ from the right, by the graph of $f$ from below and by the boundary of $\Omega_\alpha$. 

If $\alpha \in \left(\frac{\lambda^2+4\lambda-4}{2\lambda^3}, \frac{-\lambda^2+4\lambda+4}{8\lambda}\right]$, then $\mathcal{ D}_2$ splits into two parts and $\mathcal{ D}$ consists of three  components.

If  $\alpha \in \left( \frac{-\lambda^2+4\lambda+4}{8\lambda},\frac{1}{\lambda}\right)$, then $\mathcal{ D}$ consists of four  components: $\mathcal{ D}_1$, the two parts of $\mathcal{ D}_2$ and an additional part $\mathcal{ D}_3$, bounded by the line $t=r_{p-1}$, the graph of $g$ and the line $v=H_{2p-2}$; also see Figure~\ref{im: danger zones even}.
\end{Lemma}

\begin{figure}[!ht]
\includegraphics[height=7cm]{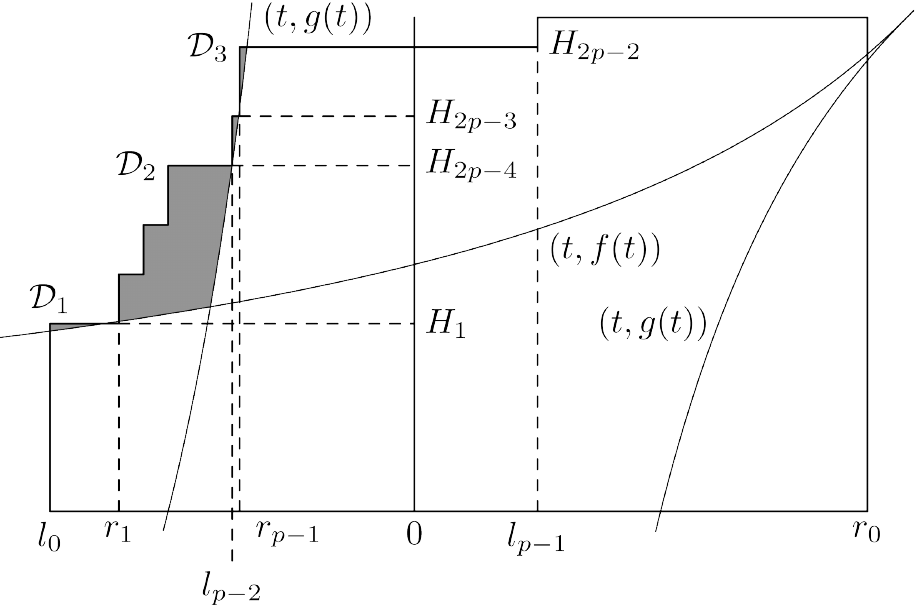}
\caption{Sketch of $\mathcal{D}$ in $\Omega_\alpha$. The number of steps on the left boundary of $\mathcal{ D}_2$ is about $p-4$. In this figure we took $\alpha \in ( \frac{-\lambda^2+4\lambda+4}{8\lambda},\frac{1}{\lambda})$, so $\mathcal{D}_2$ is split into two components and there is a region $\mathcal{D}_3$.} 
\label{im: d in example q 6}
\end{figure}
\begin{proof}
Recall that $\mathcal{H}_q=\frac12$. First assume $t \geq 0$. As in the case $q=4$, the graphs do not intersect for $t \leq r_0$. Thus every point $(t_n,v_n) \in \Omega_\alpha^+$ is below the graph of $f(t)$ or above the graph of $g(t)$. Again by~(\ref{eq: relations f g theta}), $\min \{\Theta_{n-1},\Theta_n\} <\frac12$.

Assume $t<0$. The graph of $f(t)$ intersects the line $v=H_1$ in the point $(-H_{2p-3},H_1)$ and we find that $l_0 < -H_{2p-3} < r_1$ if $\alpha < \frac{1}{\lambda}$. Since the function $f(t)$ is strictly increasing and  $f(0)=\mathcal{H}_q < \frac{1}{\lambda} = H_2$ it follows that the graph of $f(t)$ does not intersect any of the line segments $v=H_i$ for $i=2,\dots,2p-2$ for $t<0$.

Next we consider the intersection points of the graph of $g(t)$ with the line segments $v=H_i$ for $i=1,\dots,2p-2$. We work from right to left. The intersection point of the graph of $g$ and the line $v= H_{2p-2} = \frac{\lambda}{2}$ is given by $\left(\frac{-2}{\lambda+4},\frac{\lambda}{2}\right)$. The first coordinate of this point is larger than $r_{p-1}$ if and only if $\alpha>\frac{-\lambda^2+4\lambda+4}{8\lambda}$ and always smaller than $l_{p-1}$, since $l_{p-1}>0$ . We conclude that $\frac{-2}{\lambda+4}$ is in the interval $J_{2p-2} $ if and only if $\alpha \in \left(\frac{-\lambda^2+4\lambda+4}{8\lambda},\frac{1}{\lambda}\right)$. 

The intersection point of the graph of $g(t)$ with the line $v=H_{2p-3}=\lambda-1$ is given by  $(-H_1,H_{2p-3})$. Since $-\delta_1<-H_1<r_{p-1}$ we have by Theorem~\ref{th: position l r even} that $-H_1 \in J_{2p-3}$ for all $\alpha \in \left(\frac12,\frac{1}{\lambda}\right)$. 

Furthermore $g(l_{p-2})= -2 -\frac{(-\alpha\lambda^2+2\alpha+1)\lambda}{\alpha\lambda^2-2}$ and we find that $g(l_{p-2}) > H_{2p-4} = \lambda-\frac{2}{\lambda}$ if and only if  $\alpha > \frac{\lambda^2+4\lambda-4}{2\lambda^3}$. So if $\alpha \in \left( \frac{\lambda^2+4\lambda-4}{2\lambda^3},\frac{1}{\lambda}\right)$ then $\mathcal{ D}_2$ consists of two separated parts.

The graph of $g(t)$ does not intersect any of the other lines $v=H_i$ with $i=1,\dots,2p-5$, since $g$ is strictly increasing and $g(r_{p-2}) = -2 - \frac{1}{r_{p-2}} = -2 + \lambda +r_{p-1} <0$, where we used that $r_{p-2}=\frac{-1}{\lambda+r_{p-1}}$, $\lambda<2$ and $r_{p-1}<0$. \end{proof}
We see that $\mathcal{ D}$ stretches over several intervals $J_n$, which was not the case for $q=4$. Points $(t_n,v_n)$ in $\mathcal{ D}_1$ have \mbox{$t_n \in J_1 = [l_0,r_1)$}, points in $\mathcal{ D}_2$ have $t_n \in J_2 \cup \dots \cup J_{2p-1}=  [r_1,r_{p-1})$ and points in $\mathcal{ D}_3$ have $t_n \in J_{2p-2}= [r_{p-1},l_{p-1}]$.

On $\mathcal{D}$ we consider  $\Theta_{n+1}$, the next approximation coefficient. We wish to express $\Theta_{n+1}$  locally as a function of only $t_n$ and $v_n$. We divide $\mathcal{D}$ into subregions where $d_{n+1},\varepsilon_{n+1}$ and $\varepsilon_{n+2}$ are constant. This gives three regions; see Table~\ref{ConstantsEven} for the definition of the subregions.

\begin{table}[!ht]
\begin{displaymath}
\begin{array}{lrccll|rrr}
\multicolumn{6}{c|}{\textrm{Region}} & \quad d_{n+1} &\quad  \varepsilon_{n+1}
&\quad  \varepsilon_{n+2} \\
\hline
&&&&&&&&\\
\textrm{(I)} & \bigg\{ (t_n,v_n) \in \mathcal{ D}\;| \; l_0 &\leq & t_n & <&\frac{-1}{\lambda}\bigg\} & 1 & -1 & -1\\
 &&&&&&&&\\
 \textrm{(II)}& \bigg\{ (t_n,v_n) \in \mathcal{ D}\;| \; \frac{-1}{\lambda} &\leq & t_n &<&-\delta_1\bigg\}& 1 & -1 & 1\\
&&&&&&&&\\
\textrm{(III)}& \bigg\{ (t_n,v_n) \in \mathcal{ D}\;| \; -\delta_1 &\leq & t_n & <& \frac{-1}{2\lambda} \bigg\}& 2 & -1 & -1
\end{array}
\end{displaymath}
\caption{Subregions of $\mathcal D$ giving constant coefficients. }
\label{ConstantsEven}
\end{table}

We analyse $\Theta_{n+1}$ on the three regions. 

\subsection*{Region(I)}
On Region (I) we have that $\Theta_{n+1}< \frac12$ if and only if $v_n> \frac{2\lambda t_n + 2 \lambda +1}{2 \lambda t_n -t_n +2}$. 

\subsection*{Region(II)}
Region (II) is mapped to $\Omega_\alpha^+$ under $\mathcal{T}_\alpha$, so on Region (II) \mbox{$\Theta_{n+1}<\frac12$} for all points $(t_n,v_n)$.



\subsection*{Region(III)}
We denote the intersection of $\mathcal{D}_2$ and Region (III) by $\mathcal{A}$. The vertices of $\mathcal{A}$ are given by $\left( -\delta_1,g(-\delta_1)\right), \, \left(\frac{-1}{\lambda +1},\lambda-1\right)  \textrm{ and }  (-\delta_1,\lambda-1).$ The vertices of $\mathcal{D}_3$ are given by $\left( r_{p-1},g(r_{p-1})\right), \, \left(\frac{-2}{\lambda+4},\frac{\lambda}{2}\right)  \textrm{ and } (r_{p-1},\frac{\lambda}{2}).$ See Figure~\ref{im: danger zones even}.

We focus on Region (III) and discuss points in Region (I) later. 

\begin{figure}[!ht]
\includegraphics[height=50mm]{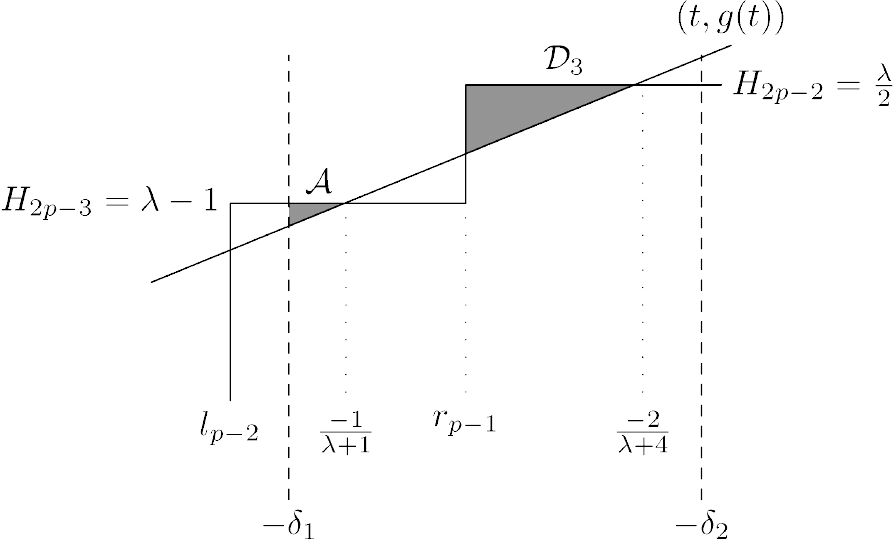}
\caption{Region (III) in $\Omega_\alpha$. If $\alpha<\frac{-\lambda^2+4\lambda+4}{8\lambda}$ there is no region $\mathcal{D}_3$.}
\label{im: danger zones even}
\end{figure}

We want to determine bounds for the minimum of three consecutive approximation coefficients on $\mathcal{A}$ and $\mathcal{D}$. 
\begin{Lemma}
\label{lem: theta's on region (III)}
For each point $(t_n,v_n)$ in Region  {\rm(III)} 
$$
\min \{\Theta_{n-1},\Theta_{n},\Theta_{n+1}\} = \Theta_{n}.
$$ 
\end{Lemma}
\begin{proof}
On region (III) it holds that $v_n > -t_n$, so from~(\ref{eq: formulae for theta}) it immediately follows that
\[\Theta_{n-1}>\Theta_{n}.\]

Using~(\ref{eq: formula Theta_n+1}) we find that $\Theta_{n+1}>\Theta_{n}$ if and only if  $v_n < \frac{(4\lambda^2-1)t_n+2\lambda }{2\lambda t_n+1}$. Put $v(t)=  \frac{(4\lambda^2-1)t+2\lambda }{2\lambda t+1}$. This function is decreasing in $t$ for $t< \frac{-1}{2\lambda}$. We find that $v\left(\frac{-2}{\lambda+4} \right)>1$ if and only if $\lambda > \frac{11+\sqrt{73}}{12}$. This last inequality is satisfied for all $\lambda_q$ with $q\geq 6$.
\end{proof}

\begin{Corollary}
For each point $(t_n,v_n)$ in Region {\rm(III)} 
$$\min \{ \Theta_{n-1},\Theta_n,  \Theta_{n-1}\} > \frac12.$$
\end{Corollary}

It follows that for every point $(t_n,v_n)$ in Region (III) we have
$$
\min \{ \Theta_{n-1},\Theta_n,\Theta_{n+1} \} \leq \max_{(t,v) \in \textrm{ Region (III)} } \Theta_n (t,v).
$$

The partial derivatives of $\Theta_n$ on Region (III) are given by 
\[
\frac{\partial \Theta_n}{\partial t_n} = \frac{-1}{(1+t_nv_n)^2}<0 \quad \textrm{ and } \quad \frac{\partial \Theta_n}{\partial v_n} = \frac{t_n^2}{(1+t_nv_n)^2}>0.
\]

We find that $\Theta_n$ takes its maximum on $\mathcal{A}$ in the upper left corner, the point $(-\delta_1,H_{2p-3})$, and on $\mathcal{D}_3$ in the vertex $\left(r_{p-1},H_{2p-2}\right)$. 
Using~(\ref{eq: lp1 rp1 middle even}) we find that these maxima are given by 
\begin{eqnarray*}
\Theta_n(-\delta_1,\lambda-1) &=& \frac{\delta_1}{1-\delta_1(\lambda-1)}= \frac{1}{\alpha \lambda +1},\\
\Theta_n\left(r_{p-1},\frac{\lambda}{2}\right) &=& \frac{-r_{p-1}}{1+\frac{r_{p-1}\lambda}{2}}= \frac{2\lambda(2\alpha-1)}{4-\lambda^2}.
\end{eqnarray*}

We find that $\Theta_n(-\delta_1,\lambda-1) > \Theta_n\left(r_{p-1},\frac{\lambda}{2}\right)$ if and only if   $\alpha < \frac{\lambda-2+\sqrt{-3\lambda^2+4\lambda+20}}{4\lambda}$. For all $\lambda<2$ we have $\frac{-\lambda^2+4\lambda+4}{8\lambda}<  \frac{\lambda-2+\sqrt{-3\lambda^2+4\lambda+20}}{4\lambda}< \frac{1}{\lambda}$

\begin{Corollary}
\label{cor: theta op d3 en a}
For every point $(t_n,v_n)$ in Region (III)
$$
\min \{ \Theta_{n-1},\Theta_{n},\Theta_{n+1} \} \leq 
\begin{cases}
\frac{1}{\alpha \lambda +1} \quad \textrm{if } \alpha \in \left( \frac12, \frac{\lambda-2+\sqrt{-3\lambda^2+4\lambda+20}}{4\lambda}\right],\\
\\
\frac{2\lambda(2\alpha-1)}{4-\lambda^2}  \quad \textrm{if } \alpha \in \left( \frac{\lambda-2+\sqrt{-3\lambda^2+4\lambda+20}}{4\lambda},\frac{1}{\lambda}\right).
\end{cases}
$$
\end{Corollary}



\subsection*{Orbit of points in Region (III)}
We study the orbit of points in $\mathcal{A}$ and $\mathcal{D}_3$ to derive the spectrum for $\alpha$-Rosen fractions. We call $p-1$ consecutive applications of $\mathcal{T}_\alpha$ a \emph{round}. We use M\"obius transformations fto find an explicit formula for $\mathcal{T}_\alpha^{p-1}$; see~\cite{Ford}. Let $S$ and $T$ be the generating matrices of the group $G_q$
\begin{equation}
\label{eq: S T matrices 3}
S \, =\, 
\left[ \begin{array}{cc}
	1   &  \lambda\\
 	0   &  1
\end{array}\right] \,
\textrm{ and } \,
 T \, =\, 
 \left[ \begin{array}{cc}
        0   &  -1\\
        1   &  0
\end{array}\right].
\end{equation}

Recall that in this context we consider matrices $M$ and $-M$ to be identical. 

\begin{Lemma}
\label{lem: mobius y from x}
Let $(t,v) \in \Omega_\alpha$ be given. Put $d=d(t) ,\,  \varepsilon=\varepsilon(t) $ and $A   \, =\, 
 \left[ \begin{array}{cc}
        -d\lambda   & \varepsilon\\
        1   &  0
\end{array}\right].$ Then 
$$\mathcal{T_\alpha}(t,v) = (A(t),TA \,T(v)).
$$
\end{Lemma}
\begin{proof}
Formula~(\ref{def: chap 3 Talpha}) gives
$$
T_\alpha(t) = \displaystyle{\frac{\varepsilon}{t}}-d\lambda =  \displaystyle{\frac{-d\lambda t+\varepsilon}{t}} =  \left[ \begin{array}{cc}
        -d\lambda   & \varepsilon\\
        1   &  0
\end{array}\right]\,(t). 
$$
Now it easily follows that
\[
TA\,T(v)= \left[ \begin{array}{cc}
        0  & 1\\
        \varepsilon   &  d \lambda
\end{array}\right](v) =\frac{1}{\varepsilon v + d \lambda}.
\]
Hence $\mathcal{T}_\alpha (t,v)=  (A(t),TA\,T(v))$ as given in  Definition~\ref{def: T(x,y)_alpha}. 
\end{proof}

\begin{Lemma}
\label{lem: m announce}
Put  $\mathcal{M}=(S^{-1}T)^{p-2} \, S^{-2}T$. For $(t,v)\in \mathcal{A} \, \cup \, \mathcal{D}_3$ we have
\[
\mathcal{T_\alpha}^{p-1}(t,v)= (\mathcal{M}(t),T\mathcal{M}T(v)).
\]
\end{Lemma}
\begin{proof}
First assume $(t,v) \in \mathcal{A}$, so $-\delta_1 \leq t \leq \frac{-1}{\lambda+1}<-\delta_2$. We have $\varepsilon(t)=-1$ and $d(t)=2$, and
$$
T_\alpha(t)= \frac{-1}{t}-2\lambda.
$$
We note that
$$
S^{-2}T (t) = \left[ \begin{array}{cc}
        -2 \lambda &-1  \\
         1 & 0
\end{array}\right](t) = \frac{-1}{t}-2\lambda.
$$
We find
$$
T_\alpha(-\delta_1) =(\alpha-1) \lambda = l_0\quad \textrm{and} \quad T_\alpha\left(\frac{-1}{\lambda+1}\right) =  1-\lambda = -H_{2p-3}.
$$
As noted in the proof of Lemma~\ref{th: shape D}, one has $-H_{2p-3} <r_1$. From Theorem~\ref{th: position l r even} and the above estimates it follows that for both $-\delta_1$ and $r_{p-1}$ the following $p-2$ applications of $T_\alpha$ give $\varepsilon=-1$ and $d=1$.  Thus we use $p-2$ times
$$
T_\alpha(t)=\frac{-1}{t}-\lambda = \frac{-\lambda t -1}{t} = \left[ \begin{array}{cc}
        -\lambda  & -1\\
       1  &  0
\end{array}\right](t)=V(t).
$$
Combining the first step with these $p-2$ steps we find $\mathcal{M}= (S^{-1}T)^{p-2}  S^{-2}T$ for points  $(t,v) \in \mathcal{A}$.  From~Lemma~\ref{lem: mobius y from x} and the fact that $T\,T=I$ we find that the second coordinate is given by $T\mathcal{M}T$.

Now assume  $(t,v) \in \mathcal{D}_3$. In this case $ \alpha \in \left( \frac{-\lambda^2+4\lambda+4}{8\lambda},\frac{1}{\lambda}\right)$ and $-\delta_1 < r_{p-1} \leq t \leq \frac{-2}{\lambda+4}<-\delta_2$. We again have $\varepsilon(t)=-1$ and $d(t)=2$, and find
$$
T_\alpha(r_{p-1}) =  \frac{2+\lambda^2(1-3\alpha)}{(2\alpha-1)\lambda}  =r_p \quad \textrm{and} \quad T_\alpha\left(\frac{-2}{\lambda+4}\right) = 2-\frac{3\lambda}{2}.
$$
For  $ \alpha \in \left( \frac{-\lambda^2+4\lambda+4}{8\lambda},\frac{1}{\lambda}\right)$ we have $r_p <  2-\frac{3\lambda}{2}< l_1$, so like before we apply $T_\alpha$ in the next $p-2$ steps with $\varepsilon=-1$ and $d=1$.
\end{proof}

We use the auxiliary sequence $B_n$ from~(\ref{eq: bn chap 3}) to find powers of $S^{-1}T$.
\begin{Lemma}
\label{lem: st explicit}
For $n \geq 1$ we have
$$
(S^{-1}T)^n \, =\, 
 \left[ \begin{array}{cc}
      -B_{n+1}  & -B_n\\
        B_n   &  B_{n-1}
\end{array}\right].
$$
\end{Lemma}
\begin{proof}
We use induction. For $n=1$ we have
$$
S^{-1}T \, =\, 
 \left[ \begin{array}{rr}
     -\lambda  & -1\\
        1   & 0
\end{array}\right] \, =\, 
 \left[ \begin{array}{cc}
      -B_{2}  & -B_1\\
        B_1   &  B_{0}
\end{array}\right].
$$
Assume that
$$
(S^{-1}T)^{n-1} \, =\, 
 \left[ \begin{array}{cc}
      -B_{n}  & -B_{n-1}\\
        B_{n-1}   &  B_{n-2}
\end{array}\right].
$$
We find
$$
(S^{-1}T)^n \, =\, 
 \left[ \begin{array}{cc}
      -B_{n}  & -B_{n-1}\\
        B_{n-1}   &  B_{n-2}
\end{array}\right] \, \left[ \begin{array}{cc}
     -\lambda  & -1\\
        1   & 0
\end{array}\right] \, =\, 
 \left[ \begin{array}{cc}
      -B_{n+1}  & -B_n\\
        B_n   &  B_{n-1}
\end{array}\right] . 
$$
\end{proof}

\begin{Lemma}
\label{lem: M T  explicit}
The function $\mathcal{T_\alpha}^{p-1}$ is explicitly given by
\[
\mathcal{T_\alpha}^{p-1}(t,v)= \frac{B_p}{2} \left(  \left[
        \begin{array}{cc}
        -\lambda^2 -2 & -\lambda \\
        \lambda^3 - \lambda & \lambda^2-2
        \end{array}
        \right]
(t), 
 \left[\begin{array}{cc}
        -\lambda^2 + 2 &\lambda^3 - \lambda \\
         - \lambda & \lambda^2+2
        \end{array}
        \right](v)\right). 
\]
\end{Lemma}
\begin{proof}
We compute $\mathcal{M}$ given in Lemma~\ref{lem: m announce} by $\mathcal{M}=(S^{-1}T)^{p-2} S^{-2}T$.

From Lemma~\ref{lem: st explicit} we find
$$
(S^{-1}T)^{p-2} \, =\, 
 \left[ \begin{array}{cc}
      -B_{p-1}  & -B_{p-2}\\
        B_{p-2}   &  B_{p-3}
\end{array}\right].
$$

Using~(\ref{eq: B_n even relations}) and~(\ref{eq: bn chap 3}) we find
$$
B_{p-1} = \frac{\lambda}{2}B_p, \quad B_{p-2}= \left(\frac{\lambda^2}{2}-1\right) B_p  \quad \textrm{and} \quad B_{p-3} = \left( \frac{\lambda^3}{2} -\frac{3\lambda}{2} \right)B_p.
$$
So
$$
(S^{-1}T)^{p-2} \, =\, 
 \frac{B_p}{2}
\left[\begin{array}{cc}
\lambda & \lambda^2-2\\
-\lambda^2 +2 & - \lambda^3 +3\lambda 
\end{array}
\right],
$$
and we find
\[
\mathcal{M}=(S^{-1}T)^{p-2}  S^{-2}T =  \frac{B_p}{2}  \left[
        \begin{array}{cc}
        -\lambda^2 -2 & - \lambda \\
        \lambda^3 - \lambda & \lambda^2-2
        \end{array}
        \right].
\]

The second coordinate is easy to calculate.
\end{proof}

With the explicit fomula for $\mathcal{M}(t)$ we can easily compute its fixed points, they are given by
$$
t_1 = \frac{-1}{\lambda+1}
\quad \textrm{and} \quad
t_2 = \frac{-1}{\lambda-1}.
$$
The fixed points of $T\mathcal{M}T(v)$ are given by
$$
v_1 = \lambda+1
\quad \textrm{and} \quad
v_2 = \lambda-1.
$$

\begin{Corollary}
\label{cor: even fixed point}
The point $\left(\frac{-1}{\lambda+1},\lambda-1 \right)$ is a fixed point of $\mathcal{T}_\alpha^{p-1}(t,v).$ 
\end{Corollary}

\begin{Rmk}
If $\varepsilon=-1$ and $d$ is constant, then $T_\alpha(t)$ is strictly increasing in $t$. From this and the above corollary we find  that for $(t,v) \in \mathcal{A}$ we have $\mathcal{M}(t)\leq t$, whilst for points  $(t,v) \in \mathcal{ D}_3$ we have $\mathcal{M}(t)> t$.

\end{Rmk}

\subsection*{Flushing}
\label{sec: flush}
We say a point is \emph{flushed} when it is mapped from $\mathcal{D}$ to a point outside of $\mathcal{D}$ by $\mathcal{T}_\alpha$. We look at the flushing of points in $\mathcal{A}$ and $\mathcal{D}_3$. 

\subsubsection*{Flushing from \texorpdfstring{$\mathcal{A} $}{A}}
Combining all the above we find that the vertices of $\mathcal{A}$ are mapped as follows under $\mathcal{T}^{p-1}$.
\begin{eqnarray*}
\left(-\delta_1, \lambda-1 \right)& \mapsto & \left(\frac{\alpha\lambda^2-2}{(-\alpha\lambda^2+2\alpha+1)\lambda},\lambda-1 \right) = (l_{p-2},\lambda-1)\\
\left(\frac{-1}{\lambda+1},\lambda-1 \right) & \mapsto& \left(\frac{-1}{\lambda+1},\lambda-1 \right), \\
\left(-\delta_1,g(-\delta_1) \right)  &\mapsto& \left(l_{p-2},\lambda-\frac{(2\alpha-1)\lambda-4}{(\alpha\lambda-2)\lambda-2}\,\right) .
\end{eqnarray*}

\begin{figure}[!ht]
\includegraphics[height=30mm]{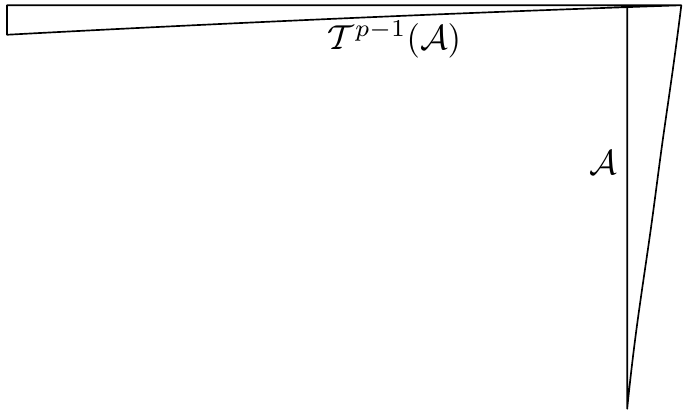}
\caption{$\mathcal{A}$ and its $(p-1)$st transformation under $\mathcal{T}_\alpha$.}
\label{im: danger normal even}
\end{figure}

The image of $\mathcal{A}$ under $\mathcal{T}^{p-1}_\alpha$ is a long, thin ``triangle" that has a ``triangular" intersection with $\mathcal{A}$; see Figure~\ref{im: danger normal even}. We notice in particular that $(-\delta_1,H_{2p-3})$ is included in $\mathcal{T}_\alpha^{p-1}(\mathcal{A})$. However, the part of $\mathcal{T}_\alpha^{p-1}(\mathcal{A})$ on the left-hand side of the line $t=-\delta_1$ is in Region (II).  So these points are flushed in the next application of $\mathcal{T}_\alpha$. 

We conclude that for the points $(t_n,v_n) \in \mathcal{A}$ with  $T_\alpha^{p-1}(t_n)<-\delta_1$ we have
\[
\min\{\Theta_{n-1},\Theta_{n},\Theta_{n+1},\dots,\Theta_{n+p-1} > \mathcal{H}_q\}  \quad \textrm{and} \quad \Theta_{n+p}< \mathcal{H}_q.
\]

The following theorem generalizes this idea to multiple rounds. Recall from~(\ref{def hat t and v even}) that we define $\tau_k$ by
$$
(\tau_k,\nu_k)= \mathcal{T}_{\alpha}^{-k(p-1)}(-\delta_1,H_1 ).
$$

\begin{Theorem}
\label{th: flushing} Let $k\geq 1$ be an integer. 

Any point $(t,v)$ of $\mathcal{A}$ is flushed after exactly $k$ rounds if and only if $$ \tau_{k-1} \leq t <
\tau_{k}.
$$

For any $x$ with $\tau_{k-1}\leq t_n < \tau_k$,
$$
\min\{\Theta_{n-1},\Theta_{n},\ldots,\Theta_{n+k(p-1)-1}, \Theta_{n+k(p-1)}\}>\mathcal{H}_q,
$$
while
$$
\Theta_{n+k(p-1)+1} <\mathcal{H}_q.
$$
\end{Theorem}
\begin{proof}
A point $(t,v) \in \mathcal{A}$ gets flushed after exactly $k$ rounds if $k$ is minimal such that $\mathcal{T}_\alpha^{k(p-1)}(t,v)$ has
its first coordinate smaller than $-\delta_1$. The result follows from  the definition of $\tau_{k}$ and the above.
\end{proof}

\subsubsection*{Flushing from \texorpdfstring{$\mathcal{D}_3 $}{D3}}

Recall that the vertices of $\mathcal{ D}_3$ are given by  $\left( r_{p-1},g(r_{p-1})\right)$, $\left(\frac{-2}{\lambda+4},\frac{\lambda}{2}\right)$  and $(r_{p-1},\frac{\lambda}{2})$; see also Figure~\ref{im: danger zones even}. We have $\frac{-1}{\lambda+1}< r_{p-1} <\frac{-2}{\lambda+4}$ and $\lambda-1 <g(r_{p-1}) < \frac{\lambda}{2}$.  The fixed point $\left(\frac{-1}{\lambda+1},\lambda-1 \right)$ of $\mathcal{T}^{p-1}$ is repelling in the $t$-direction and attractive in the $v$-direction. 

\begin{Lemma}
\label{lem: d flushen}
There exists a positive integer $K$ such that all points in $\mathcal{ D}_3$ are flushed after $K$ rounds.
\end{Lemma}
\begin{proof}
Let $(t_1,\lambda-1+v_1)$ be a point in $\mathcal{ D}_3$, it follows that $0< v_1 < 1$. With Lemma~\ref{lem: M T  explicit} we find for the second coordinate
$$
\pi_2(\mathcal{T}^{p-1}_\alpha(t_1,\lambda-1+v_1)) = \frac{(-\lambda^2+2)(\lambda-1+v_1)+\lambda^3-\lambda}{-\lambda(\lambda-1+v_1)+\lambda^2+2} = \lambda-1+\frac{2-\lambda }{\lambda+2-\lambda v_1} v_1.
$$ 
As $\frac{2-\lambda }{\lambda+2-\lambda v_1} < 1 - \frac{\lambda}{2}<1$, the result follows.
\end{proof}

\begin{Rmk}
We could divide $\mathcal{ D}_3$ in parts that get flushed after $1,2,\dots,K$ rounds, respectively. As we saw in the example for $q=4$ the formulas needed to do this are rather ugly and in this general case they only get worse. For our main result we only need that after finitely many rounds all points are flushed out of $\mathcal{ D}_3$.  
{\phantom{xx}}
\end{Rmk}

We still need to consider points in Region (I). It follows from Theorem~\ref{th: position l r even} that after at most $p-2$ steps each such point is either flushed or mapped into $\mathcal{A} \cup \mathcal{ D}_3$. We are now ready to prove Theorem~\ref{th: tong even general} for this case.
 
\emph{Proof of Theorem~\ref{th: tong even general} for $\alpha \in \left(\frac{1}{2},\frac{1}{\lambda}\right)$.}
By definition $\tau_k < \tau_{k+1}$.  In~\cite{BKS} it was shown that $\frac{-1}{\lambda+1}=[\overline{(-1:2),(-1:1)^{p-2}}]$ and we find $\displaystyle{\lim_{k\rightarrow \infty} \tau_k }= \frac{-1}{\lambda+1}$. Recall that $\displaystyle{c_{k}=\frac{-\tau_{k-1}}{1+\tau_{k-1}\nu_{k-1}}}$. It follows that  $c_{k+1} <c_{k} $ and $\displaystyle{\lim_{k\rightarrow\infty} c_k= \frac{1}{2}}$.

Take an integer $k$ such that $k>K$ from Lemma~\ref{lem: d flushen}. Take a point $(t_n,v_n) \in \mathcal{ D}$ that did not get flushed in the first $k-1$ rounds. There exists an index $i$ with $0 \leq i \leq p-2$ such that $(t_{n+i},v_{n+i})$ is either in $\mathcal{A}$ or flushed. We assume $(t_{n+i},v_{n+i})\in \mathcal{A}$, otherwise we are done. From Theorem~\ref{th: flushing} we find that $t_{n+i}\geq \tau_k$. Thus $\Theta_n(t_{n+i},v_{n+i}) \leq \Theta_n(\tau_{k-1},\lambda-1) = c_k.$ 
\hfill $\Box$\medskip\

\subsection{Even case for \texorpdfstring{$\alpha = \frac{1}{\lambda}$}{alpha is 1/lambda}}
\label{sec: even lambda}
The natural extension $\Omega_{1/\lambda}$ can be found from  $\Omega_{1/2}$ by mirroring in the line $v=-t$ if $t\leq 0$, and in the line $v=t$ if $x\geq 0$; see [DKS]. 

In the case $\alpha=\frac12$ we have $l_n =r_n$ for $n \geq 1 $; see Theorem~\ref{th: position l r even}. We put $\varphi_0=l_0 = - \frac{\lambda}{2}$ and denote $\varphi_n = l_n=r_n =T_{1/2}^n\left( \varphi_0\right)$. Put
$$
L_1 = \frac{1}{\lambda + 1}  \quad \textrm{and} \quad L_n =\frac{1}{\lambda - L_{n-1}} \quad \textrm{for} \quad n=2,3,\dots,p-1.\\
$$ 

We know from~\cite{DKS} that
$$
\begin{aligned}
\Omega_{1/2} &= \left( \bigcup_{n=1}^{p-1} [\varphi_{n-1},\varphi_n) \times [0,L_n] \right) \cup \left[ 0, -\varphi_0 \right) \times [0,1],\\
\Omega_{1/\lambda} &= \left( \bigcup_{n=1}^{p-2} [-L_{p-n},-L_{p-n-1})  \times  [0,-\varphi_{p-n-1}] \right)  \cup [-L_1,1) \times  [0,-\varphi_0],
\end{aligned}
$$
see Figure~\ref{im: mirroring} for an example with $q=8$.


\begin{figure}[!ht]
\includegraphics[height=6cm]{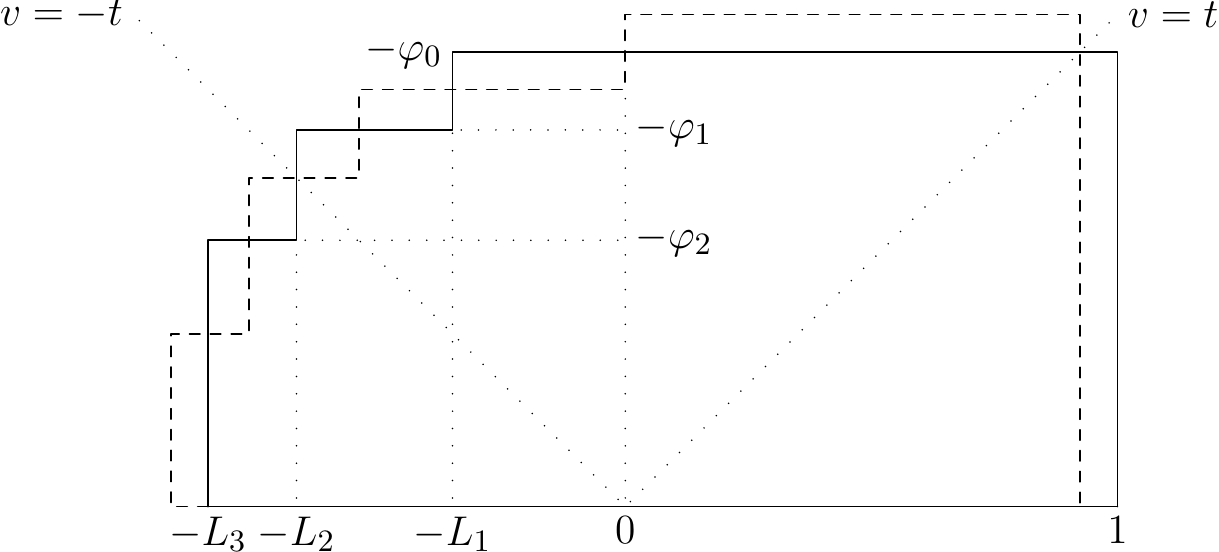}
\caption{We find $\Omega_{\frac{1}{\lambda}}$ by mirroring $\Omega_{\frac12}$  (dashed) in  $v= |t|$. In this example $q=8$.}
\label{im: mirroring}
\end{figure}

We now could proceed as in the previous two subsections; For $t\geq 0$ one easily sees that the graphs of $f$ and $g$ do not meet in $\Omega_{1/\lambda}$ (they meet in $(1,1)$, which is outside $\Omega_{1/\lambda}$). As before, from this it follows that for $t_n\geq 0$ we have that $\min \{ \Theta_{n-1},\Theta_n\} <\frac{1}{2}$. So we only need to focus on the region ${\mathcal D}$, and how it is eventually ``flushed.'' However, we can also derive the result directly from the case $\alpha=\frac12$.

Define the map $M : \Omega_{1/\lambda} \to \Omega_{1/2}$ by 
\begin{equation}
\label{eq: M mirror}
M(t,v)=\begin{cases}
(-v,-t) \qquad & \textrm{if }  t <0,\\
(v,t) \qquad &\textrm{if }  t  \geq 0.\\
\end{cases}
\end{equation}

In~\cite{DKS} it was shown that 
\begin{equation}
\label{eq: isomorphism }
\mathcal{T}_{1/\lambda} (t,v) = M^{-1}\left(\mathcal{T}^{-1}_{1/2}(M(t,v))\right).
\end{equation}

This implies that the dynamical systems $(\Omega_{1/2},\mu_{1/2},{\mathcal T}_{1/2})$ and $(\Omega_{1/\lambda},\mu_{1/\lambda},{\mathcal T}_{1/\lambda})$  are isomorphic. These systems ``behave dynamically in the same way'', also see~\cite{KSS2}. Usually, this is not much of help if we want to obtain Diophantine properties of one system from the other system. But the special form of the isomorphism $\mathcal{M}$ makes it possible to prove directly that these systems posses the same ``Diophantine properties.'' Essentially, if one ``moves forward in time" in $(\Omega_{1/2},\mu_{1/2},{\mathcal T}_{1/2})$, then one ``moves backward in time" in $(\Omega_{1/\lambda},\mu_{1/\lambda},{\mathcal T}_{1/\lambda})$  and vice versa. 

\begin{Theorem}
\label{th: sequence theta dual }
Let $\ell\in\N$, and let $\Theta_{n-1},\Theta_n,\dots,\Theta_{n+\ell}$ be $\ell +2$ consecutive approximation coefficients of the point $(t_n,v_n)\in\Omega_{1/\lambda}$, then there exists a point $({\tilde t}_m,{\tilde v}_m)\in \Omega_{1/2}$ and approximation coefficients ${\tilde \Theta}_m, {\tilde \Theta}_{m-1},\dots,{\tilde \Theta}_{m-\ell -1}$, given by
$$
\begin{aligned}
&{\tilde \Theta}_m=\Theta_m({\tilde t}_m,{\tilde v}_m)=\frac{|{\tilde t}_m|}{1+{\tilde t}_m{\tilde v}_m}, \quad {\tilde \Theta}_{m-1}=\Theta_{m-1}({\tilde t}_m,{\tilde v}_m)=\frac{{\tilde v}_m}{1+{\tilde t}_m{\tilde v}_m}, \dots, \\
& {\tilde \Theta}_{m-\ell -1}=\Theta_{m-\ell -1}({\tilde t}_{m-\ell},{\tilde v}_{m-\ell})=\frac{{\tilde v}_{m-\ell}}{1+{\tilde t}_{m-\ell}{\tilde v}_{m-\ell}},
\end{aligned}
$$
such that
\begin{equation}\label{EqualityOfThetas}
\Theta_{n-1}={\tilde \Theta}_m, \: \Theta_n={\tilde \Theta}_{m-1}, \: \Theta_{n+1}={\tilde \Theta}_{m-2},\:  \dots, \:  \Theta_{n+\ell}={\tilde \Theta}_{m-\ell -1} .
\end{equation}
\end{Theorem}
\begin{proof}
Let $({\tilde t}_m,{\tilde v}_m)=M(t_n,v_n)\in \Omega_{1/2}$ be the point in $\Omega_{1/2}$ corresponding to $(t_n,v_n)\in\Omega_{1/\lambda}$ under the isomorphism $M$~from~(\ref{eq: M mirror}). Then
$$
({\tilde t}_m,{\tilde v}_m)= \left\{
\begin{array}{ll}
(-v_n,-t_n) &\textrm{if } t_n<0,\\
 & \\
(v_n,t_n)&\textrm{if } t_n\geq 0,
\end{array}\right.
$$
and we find that
$$
{\tilde \Theta}_{m-1}=\Theta_{m-1}({\tilde t}_m,{\tilde
v}_m)=\frac{{\tilde v}_m}{1+{\tilde t}_m{\tilde v}_m}
= \frac{|t_n|}{1+t_nv_n} = \Theta_n.
$$
Similarly,
$$
{\tilde \Theta}_{m}=\Theta_{m}({\tilde t}_m,{\tilde
v}_m)=\frac{|{\tilde t}_m|}{1+{\tilde t}_m{\tilde v}_m}
= \frac{v_n}{1+t_nv_n} = \Theta_{n-1}.
$$
So we have $(\Theta_{n-1},\Theta_n)=({\tilde \Theta}_{m},{\tilde
\Theta}_{m-1})$.\medskip\

Furthermore, by~(\ref{eq: isomorphism }) we have that
\begin{eqnarray*}
(t_{n+1},v_{n+1}) &=& {\mathcal T}_{1/\lambda}(t_n,v_n) = M^{-1}\left( {\mathcal T}_{1/2}(M(t_n,v_n))\right) \\
&= &M^{-1}\left( {\mathcal T}_{1/2}({\tilde t}_m,{\tilde
v}_m)\right) 
= M^{-1}({\tilde t}_{m-1},{\tilde v}_{m-1}),
\end{eqnarray*}
and we see that
$$
({\tilde t}_{m-1},{\tilde v}_{m-1}) = M\left( t_{n+1},v_{n+1}\right)
= \left\{ \begin{array}{ll}
(-v_{n+1},-t_{n+1}) &\textrm{if } t_{n+1}<0\\
 & \\
(v_{n+1},t_{n+1}) &\textrm{if } t_{n+1}\geq 0.
\end{array}\right.
$$
But then we have that
$$
{\tilde \Theta}_{m-2}=\frac{{\tilde v}_{m-1}}{1+{\tilde t}_{m-1}{\tilde v}_{m-1}} 
= \frac{|t_{n+1}|}{1+t_{n+1}v_{n+1}} = \Theta_{n+1},
$$
and by induction it follows that
$$
\Theta_{n-1}={\tilde \Theta}_m, \: \Theta_n={\tilde \Theta}_{m-1},\:  \Theta_{n+1}={\tilde \Theta}_{m-2},\: \dots, \: \Theta_{n+\ell}={\tilde \Theta}_{m-\ell -1}.
$$
\end{proof}

\begin{Lemma}
\label{lem: t_n from theta_n}
Let $\Theta_{n-1}$ and $\Theta_n$ be two consecutive approximation coefficients of the point $(t_n,v_n)$. Then
$$
t_n=\frac{1+ \varepsilon_{n+1} \sqrt{1-4\varepsilon_{n+1} \Theta_{n-1}\Theta_n}}{2\Theta_{n-1}}\quad
{\mbox{and}}\quad
v_n=\frac{\varepsilon_{n+1}+\sqrt{1-4\varepsilon_{n+1}  \Theta_{n-1}\Theta_n}}{2\Theta_{n}}.
$$
\end{Lemma}
\begin{proof}
From~(\ref{eq: formulae for theta}) we have $\Theta_{n-1} = \displaystyle{\frac{v_n}{1+t_nv_n}}$ and $\displaystyle{\Theta_n=\frac{\varepsilon_{n+1}t_n}{1+t_nv_n}}$. It follows that 
\begin{equation}
\label{eq: vn in tn and theta}
v_n = \frac{\varepsilon_{n+1}  \Theta_{n-1}\, t_n}{\Theta_n},
\end{equation}
and substituting~(\ref{eq: vn in tn and theta}) in the formula for $\Theta_n$ yields 
$$
\varepsilon_{n+1} \Theta_{n-1} t_n^2 -\varepsilon_{n+1} t_n +\Theta_n =0,
$$
from which we find $t_n$. Substituting $t_n$ in~(\ref{eq: vn in tn and theta}) yields $v_n$. 
\end{proof}

\emph{Proof of Theorem~\ref{th: tong even general} for $\alpha = 1/\lambda$.}
Let $x$ be a $G_q$-irrational with $1/\lambda$-expansion $[\,\varepsilon_1:d_1,\varepsilon_2:d_2,\varepsilon_3:d_3,\dots \,] $ and let  $n\geq 1$ be an integer. Assume there is a $k\in\N$ is such, that
$$
\min \{ \Theta_{n-1},\Theta_n,\dots,\Theta_{n+k(p-1)} \} >
\frac{1}{2} ,
$$
otherwise we are done. From Lemma~\ref{lem: t_n from theta_n} we find the appropriate $(t_n,v_n) \in \Omega_{1/\lambda}$ for this sequence of approximation coefficients. From Theorem~\ref{th: sequence theta dual } if follows that we can find  $(\tilde{t}_m,\tilde{v}_m)\in \Omega_{1/2}$ such that 
$$
\Theta_{n-1}={\tilde \Theta}_m, \Theta_n={\tilde \Theta}_{m-1}, \dots, \Theta_{n+k(p-1)}={\tilde
\Theta}_{m-k(p-1)-1} .
$$
It follows from Theorem~\ref{th: tong even} for $\alpha=\frac12$ that
$$
\min \{ \Theta_{n-1},\Theta_n,\dots,\Theta_{n+k(p-1)} \} = \min \{ {\tilde
\Theta}_{m-k(p-1)-1},\dots,  {\tilde \Theta}_{m-1},{\tilde \Theta}_m, \} < c_k,
$$
where $c_k$ is defined in (\ref{def hat t and v even}) and Theorem~\ref{th: tong even general}. This proves Theorem~\ref{th: tong even general} for the case $\alpha=1/\lambda$. 
\hfill $\Box$\medskip\

\section{Tong's spectrum for odd \texorpdfstring{$\alpha$}{alpha}-Rosen fractions}
\label{sec:odd}
Let $q=2h+3$ for $h\geq 1$ and define
\begin{equation}
\label{def: rho}
\rho =\frac{\lambda -2 +\sqrt{\lambda^2-4\lambda +8}}{2}.
\end{equation}
We often use the following relations for $\rho$
$$
\rho^2 +(2-\lambda)\rho-1 =0 \quad
\textrm{and} \quad\frac{\rho}{\rho^2+1}  = \mathcal{H}_q = \frac{1}{\sqrt{\lambda^2-4\lambda+8}}.
$$

\begin{Rmk}
For $q=3$ (i.e.~$\lambda=1$) we are in the ``classical" case of Nakada's $\alpha$-expansions~\cite{Nakada}. In this case $\frac{\rho}{\lambda}=\frac{\sqrt{5}-1}{2}= g$ and $\mathcal{H}_q = \frac{1}{\sqrt{5}}$; see also~\cite{HK} for a discussion of this case.

\end{Rmk}



For a fixed $\lambda$ we define the following constants 
$$
\begin{aligned}
\alpha_1 & = \frac{(\lambda-2)\mathcal{H}_q+1}{\lambda},\\
\alpha_2 &= \frac{-\lambda+\sqrt{5\lambda^2-4\lambda+4}}{2\lambda},\\
\alpha_3 &= \frac{(2-\lambda)^2\mathcal{H}_q+2\lambda}{4\lambda},\\
\alpha_4& =\frac{(2-\lambda)\mathcal{H}_q-2}{(\lambda\mathcal{H}_q-2\mathcal{H}_q-2)\lambda}.
\end{aligned}
$$
For all admissible $\lambda$ we have
\begin{equation}
\label{eq: order odd alpha}
\frac12 < \alpha_1 <\alpha_2 <\alpha_3 < \frac{\rho}{\lambda} <\alpha_4 <\frac{1}{\lambda}.
\end{equation}

\begin{Rmk}
These numbers are very near to each other. For example, if $q=9$,
$$
\begin{aligned}
\alpha_1 &\approx 0.500058, \quad &\alpha_2 \approx 0.500515, \quad &\alpha_3 \approx 0.500966 \\
 \frac{\rho}{\lambda} &\approx 0.500967, & \alpha_4 \approx 0.500994,\quad & \frac{1}{\lambda}\approx 0.532089.
\end{aligned}
$$

\end{Rmk}

In~\cite{DKS} it was shown that there are four subcases for the natural extension of odd $\alpha$-Rosen fractions: $\alpha=\frac12, \alpha \in \left( \frac12, \frac{\rho}{\lambda}\right),\alpha =\frac{\rho}{\lambda}$ and $\alpha \in \left(  \frac{\rho}{\lambda}, \frac{1}{\lambda}\right] $. Again, the case $\alpha=\frac12$ had been dealt with in~\cite{KSS} and we give the details for the other cases in this section. 

The following theorem from~\cite{DKS} is the counterpart for the odd case of Theorem~\ref{th: position l r even} about the ordering of the $l_n$ and $r_n$.
\begin{Theorem}
\label{th: position l r odd}
Let $q=2h+3, h \in N, h\geq 1$. We have the following cases.
\begin{itemize}
\item[$ \alpha=\frac12:\quad$] $l_0 <r_{h+1} =l_{h+1} <r_1=l_1 <\ldots< r_{h+n} = l_{h+n} < r_n = l_n < \ldots <r_{2h-1}=l_{2h-1}<r_{h-1}=l_{h-1}<r_{2h} =l_{2h}<-\delta_1 <r_h =l_h <-\delta_2 <r_{2h+1}=l_{2h+1}=0<r_0.$\\

\item[$ \frac12 < \alpha < \frac{\rho}{\lambda}:\quad$] $l_0 <r_{h+1} <l_{h+1} <r_1 < l_1  <\ldots< r_{h+n} < l_{h+n} < r_n < l_n < \ldots<r_{2h-1}<l_{2h-1}<r_{h-1} <l_{h-1}<r_{2h}<l_{2h}<-\delta_1 <r_h <l_h < r_{2h+1}<0<l_{2h+1}<r_0.$ \\
Furthermore, we have $l_h <-\delta_2$,  $l_{2h+2}=r_{2h+2}$ and $d_{2h+2}(r_0)=d_{2h+2}(l_0)+1$.\\

\item[$\alpha=\frac{\rho}{\lambda}: \quad$] $l_0 = r_{h+1} <l_{h+1} = r_1 < \ldots < l_{n-1} = r_{h+n} < l_{h+n}=r_n <\ldots < l_{h-1}  = r_{2h}< l_{2h}= r_h = -\delta_1 <l_h = r_{2h+1}< -\delta_2 <0 <r_0$.\\

\item[$\frac{\rho}{\lambda}< \alpha < \frac{1}{\lambda}:\quad$] $l_0<r_1< l_1< r_2< \ldots <l_{h-1} < r_h <-\delta_1 <l_h <0 < r_{h+1}  <r_0.$\\
Furthermore, we have $l_{h+1} = r_{h+2}$ and $d_{h+1}(l_0)=d_{h+2}(r_0)+1$.\\
\item[$\alpha = \frac{1}{\lambda}:\quad$] $l_0 =r_1< l_1= r_2 <\ldots <l_{h-1} = r_h <-\delta_1 <l_h  =0 = r_{h+1} <r_0$. 
\end{itemize}
\end{Theorem}
\begin{Rmk}
In an preliminary version of~\cite{DKS} there was a small error in the above theorem in the case $ \frac12 < \alpha < \frac{\rho}{\lambda}$: it stated that $-\delta_2 <r_{2h+1}$. But this is only true if $\alpha <\alpha_2$. For all $ \frac12 < \alpha < \frac{\rho}{\lambda}$ one has $r_{2h+1} = -\frac{(2\alpha-1)\lambda}{\alpha\lambda^2-2\lambda+2}$, so 
$$
\begin{aligned}
r_{2h+1} \geq -\delta_2 &\Leftrightarrow -\frac{(2\alpha-1)\lambda }{\alpha\lambda^2-2\lambda+2} \geq \frac{-1}{(\alpha+2)\lambda}\\
&\Leftrightarrow \frac{-\lambda-\sqrt{5\lambda^2-4\lambda+4}}{2\lambda }< \alpha \leq \frac{-\lambda+\sqrt{5\lambda^2-4\lambda+4}}{2\lambda }=\alpha_2.
\end{aligned}
$$
We conclude that $r_{2h+1} \geq -\delta_2$ if $\alpha \in \left(\frac12,a_2\right)$ and that  $r_{2h+1} < -\delta_2$ if $\alpha \in \left[a_2,\frac{\rho}{\lambda}\right)$.
\phantom{k}
\end{Rmk}

In this section we prove the following result. 
\begin{Theorem}
\label{th: 3 tong odd general} 
Fix an odd $q = 2h+3$, with $h \ge 1$. 

(i) Let $\alpha \in \left[ \frac12, \frac{\rho}{\lambda} \right]$. Then there exists a positive integer $K$ such that for
every $G_q$-irrational number $x$ and all positive $n$ and $k >K$,
$$
\min \{ \Theta_{n-1},\Theta_n,\ldots,\Theta_{n+k(2h+1)}\} < c_{k}, 
$$
for certain constants $c_k$ with $c_{k+1} < c_{k}$ and $\displaystyle{\lim_{k \rightarrow \infty} c_{k}} = \mathcal{H}_q$.

(ii) Let $\alpha \in \left( \frac{\rho}{\lambda},\frac{1}{\lambda}\right]$.  For
every $G_q$-irrational number $x$ and all positive $n$,  one has
$$
\min \{ \Theta_{n-1},\Theta_n,\ldots,\Theta_{n+(3h+2)}\} < \mathcal{H}_q.  
$$
\end{Theorem}

Case (i) is similar to the even case, but case (ii) yields a finite spectrum. We note such a finite spectrum was already described for the case $q=3$  in~\cite{HK}.

For odd $q = 2h+3$ we define a \emph{round} by $2h+1$ consecutive applications of $\mathcal{T}_\alpha$.

\subsection*{Intersection of the graphs of \texorpdfstring{$f(t)$ and $g(t)$}{f(t) and g(t)}}
\label{intersection}
We start by looking at the behavior of the graphs of $f(t)$ and $g(t)$ as given in~(\ref{eq: def f and g}) on $\Omega_\alpha$ for odd $q$.

\begin{figure}[!ht]
\includegraphics[height=5cm]{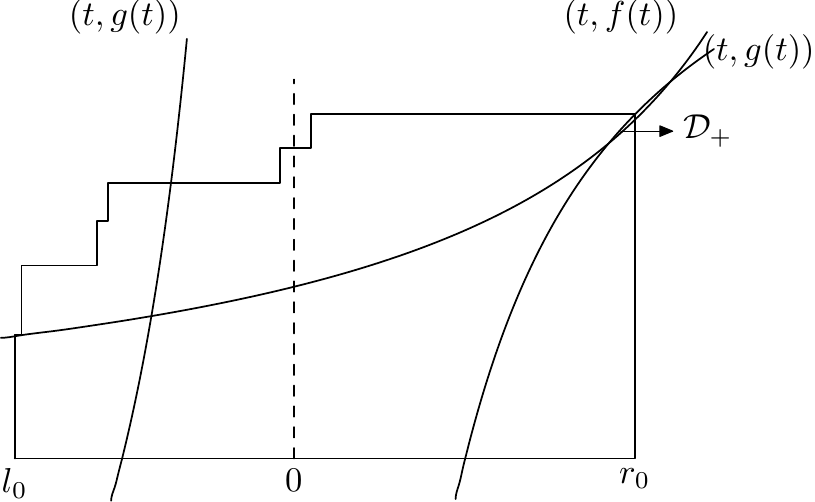}
\caption{For  $\alpha>\rho/\lambda$ we have a part ${\mathcal D}^+$ for positive $t$.}
\label{im: odd f and g}
\end{figure}

In case $t>0$ it is interesting to determine for which $t$ we have $f(t)=g(t)$. We find that
$$
f(t)=g(t)\quad  \textrm{if and only if} \quad t=\frac{1\pm \sqrt{1-4{\mathcal H}_q^2}}{2{\mathcal H}_q}.
$$
Clearly, 
$$
\frac{1+ \sqrt{1-4{\mathcal H}_q^2}}{2{\mathcal H}_q}>1\geq \alpha\lambda
$$
for all $\alpha\in [1/2,1/\lambda ]$, so we only need to consider $
\frac{1-\sqrt{1-4{\mathcal H}_q^2}}{2{\mathcal H}_q} 
 = \rho.$

Note that $f(\rho)=\rho$. In case
\begin{equation}\label{alpha_b}
\alpha = \frac{1 - \sqrt{1-4{\mathcal H}_q^2}}{2\lambda{\mathcal H}_q} = \frac{\rho}{\lambda},
\end{equation}
we find that for $t>0$ the intersection of the graphs of $f$ and $g$ is on the boundary of 
$\Omega_{\alpha}$, so if $\alpha\in [1/2, \rho/\lambda )$ we have that the intersection point for $t>0$
 is outside $\Omega_{\alpha}$, while for $\alpha\in (\rho/\lambda,1/\lambda]$ 
it will be inside  $\Omega_{\alpha}$. Therefore, for the latter values of $\alpha$ we have an extra
part ${\mathcal D}^+$ for positive $t$; see also Figure~\ref{im: odd f and g}.


\subsection{Odd case for \texorpdfstring{$\alpha \in  (\frac12, \frac{\rho}{\lambda})$}{alpha in the left interval}}
In this case the natural extension is given by $\Omega_\alpha=\bigcup_{n=1}^{4h+3} J_n \times [0,H_n]$, with
$$
\begin{aligned}
J_{4n-3}&=[l_{n-1},r_{h+n}), &J_{4n-2} &= [r_{h+n},l_{h+n})  &\quad \textrm{for} \quad n=1,\dots,h+1,\\
J_{4n-1}&=[l_{h+n},r_n),  &J_{4n} &= [r_n,l_n) &\quad \textrm{for} \quad n=1,\dots h,
\end{aligned}
$$
and
$$
\begin{aligned}
&H_1 = \frac{1}{\lambda +1/\rho}, \quad H_2 = \frac{1}{\lambda+1}, \quad H_3 = \frac{1}{\lambda+\rho},  \quad H_4 = \frac{1}{\lambda}, \\
\textrm{and} \quad &H_n = \frac{1}{\lambda - H_{n-4}} \quad \textrm{for} \quad n=5,6,\dots,4h+3.
\end{aligned}
$$

In~\cite{DKS} is shown that $H_{4h-1}=\lambda - \frac{1}{\rho},  H_{4h}=\lambda-1,  H_{4h+1}=\lambda -\rho$  and  $ H_{4h+2}=\frac{\lambda}{2}.$ Furthermore, from~\cite{DKS} we have that
$$
\begin{aligned}
l_h &= \frac{1-\alpha \lambda }{(\lambda-1)\alpha\lambda -1},  \quad r_h = -\frac{1-(1-\alpha)\lambda}{1-(1-\alpha)\lambda(\lambda-1)} \quad  \textrm{and}\\
r_{2h+1} &= -\frac{(2\alpha-1)\lambda}{\alpha\lambda^2-2\lambda+2}.
\end{aligned}
$$

We define ${\mathcal D}$ as in~(\ref{def:D}). We have proved above that in case $\alpha \leq \rho/\lambda$ we have ${\mathcal D}^+=\emptyset$. We could divide $\mathcal{ D}^-$ into regions where $d_n,\varepsilon_{n+1}$ and $\varepsilon_{n+2}$ are constant, as we did in Section~\ref{sec: even interval}. However, like before the region where $d_n = 2$ is the crucial one and we only describe the part of $\mathcal{D}^-$ that is on the right hand side of the line $t=-\delta_1$.

\begin{figure}[!ht]
\includegraphics[height=42mm]{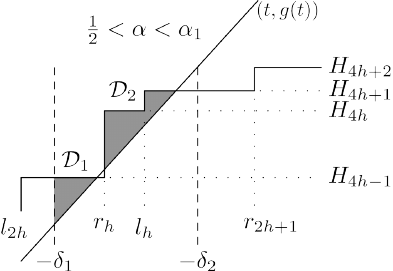}\hskip5mm
\includegraphics[height=48mm]{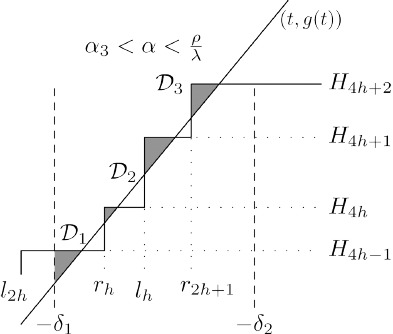}
\caption{Schematic presentation of the part of $\mathcal{ D}^-$ with $t>-\delta_1$  in two cases. If $\alpha_1 \leq \alpha < \alpha_2$, we have a similar picture as on the left: the only difference is that $\mathcal{ D}_2$ is split into two  component in this case. If $\alpha_2<\alpha \leq \alpha_3$ the picture is similar to the one on the right, but there is no $\mathcal{ D}_3$ in this case.}
\label{im: images of d in odd first case}
\end{figure}

\begin{Lemma}
\label{lem: shape d odd interval}
 If $\alpha \in \left(\frac12,\alpha_1\right)$, then the part of $\mathcal{ D}$ with $-\delta_1 <t<0$ consists of two  components: $\mathcal{ D}_1$ bounded by the line $t=-\delta_1$ from the left, the line $v=H_{4h-1}$ from above and the graph of $g$ from below. The second component $\mathcal{ D}_2$ is  bounded by line segments with $t=r_h, v=H_{4h}, t=l_h, v=H_{4h+1}$ and by the graph of $g$.  

 If $\alpha \in \left[\alpha_1,\alpha_3\right]$, then $\mathcal{ D}_1$ is as in the above case, but $\mathcal{ D}_2$ is split into two components bounded by line segments as above and by the graph of $g$.

 If $\alpha \in \left(\alpha_3, \frac{\rho}{\lambda}\right)$, then $\mathcal{ D}_1$ and $\mathcal{ D}_2$ are as in the previous case, but there is an extra component $\mathcal{ D}_3$, bounded by the line $t=r_{2h+1}$ from the left, the line $v=H_{4h+2}$ from above and the graph of $g$ from below.
\end{Lemma}
\begin{proof}
Recall that $H_{4h-1}=\lambda - \frac{1}{\rho},  H_{4h}=\lambda-1,  H_{4h+1}=\lambda -\rho$  and  $ H_{4h+2}=\frac{\lambda}{2}.$ First of all $g(l_h) = -\frac{1}{\mathcal{H}_q}-\frac{1}{l_h} =-\frac{1}{\mathcal{H}_q}-\frac{(\lambda-1)\alpha\lambda-1}{1-\alpha\lambda} $ and we find
$$
\begin{aligned}
H_{4h-1}< g(l_h) < &\quad H_{4h}  & \quad \textrm{if} \quad \frac12 < \alpha<\alpha_1,\\
H_{4h}\leq g(l_h) < &\quad H_{4h+1}  & \quad \textrm{if} \quad \alpha_1 \leq \alpha<\frac{\rho}{\lambda}.
\end{aligned}
$$
From $g(r_{2h+1}) = -\frac{1}{\mathcal{H}_q}+\frac{\alpha\lambda^2-2\lambda+2}{(2\alpha-1)\lambda} $  we find that 
$$
g(r_{2h+1}) > H_{4h+2}  \quad \textrm{if and only if} \quad \alpha < \frac{(2-\lambda)^2\mathcal{H}_q+2\lambda}{4\lambda }=\alpha_3. 
$$

Finally, for all $\alpha \in \left(\frac12,\frac{\rho}{\lambda}\right)$ we have that
$$
g(-\delta_1) < H_{4h-1} < g(r_h) < H_{4h} , \quad g(r_{2h+1})> H_{4h+1} \quad \textrm{and} \quad g(-\delta_2) > H_{4h+2},
$$
which finishes the proof.
\end{proof}

%

%



\emph{The proof of Theorem~\ref{th: 3 tong odd general} for $\frac12 \leq \alpha < \frac{\rho}{\lambda}$}.
Consider the intersection point from the graph of $g$ with the line $v=H_{4h+1}=\lambda-\rho$. The first coordinate of this point is given by
$$
t_1 = \frac{-\mathcal{H}_q}{1+(\lambda-\rho)\mathcal{H}_q} = \frac{-\rho}{1+\rho\lambda}.
$$
Note that $t_1 \in (l_h,-\delta_2)$. We find $T_\alpha(t_1) = \frac{1}{\rho}-\lambda$ and it easily follows that we have  $T_\alpha(t_1) \in (l_{h+1},r_1)$ for all $\alpha<\frac{\rho}{\lambda}$. From Theorem~\ref{th: position l r odd} we conclude that $T_\alpha^{h}(t_1) = (S^{-1}T)^{h-1}S^{-2}T(t_1)$. From Lemma~\ref{lem: st explicit} and relations~(\ref{eq: bn chap 3})~and~(\ref{eq: B_n odd relations}) we get
$$
\begin{aligned}
 (S^{-1}T)^{h-1}S^{-2}T &=   
 \left[ \begin{array}{cc}
-B_h & -B_{h-1}\\
B_{h-1} & B_{h-2}       
\end{array} \right]
\left[ \begin{array}{cc}
-2\lambda  & -1\\
1 & 0       
\end{array} \right]
\\&= B_{h+1} 
 \left[ \begin{array}{cc}
-\lambda + 1& -\lambda^2+\lambda+1\\
\lambda^2-\lambda-1 &\lambda^3 -\lambda^2-2\lambda+1      
\end{array} \right]
\left[ \begin{array}{cc}
-2\lambda  & -1\\
1 & 0       
\end{array} \right]\\
&=  B_{h+1} 
\left[ \begin{array}{cc}
\lambda^2-\lambda+1 & \lambda-1  \\
-\lambda^3+\lambda^2+1 & -\lambda^2+\lambda+1
\end{array} \right].
\end{aligned}
$$
We find $T_\alpha^{h}(t_1) = \displaystyle{- \frac{\lambda-\rho-1}{\lambda(\lambda-\rho-1)+\rho-1}}$, and using $\rho^2 +(2-\lambda)\rho-1 = 0$ we find
$$
g(T_\alpha^{h}(t_1)) = \frac{-1}{\mathcal{H}_q} - \frac{1}{T_\alpha^{h}(t_1)} =  -\rho - \frac{1}{\rho} + \lambda + \frac {\rho-1}{\lambda-\rho-1}= \lambda  -\frac{1}{\rho} = H_{4h-1}.
$$
So $(T_\alpha^{h}(t_1),\lambda  -\frac{1}{\rho} )$ is the intersection point of the graph of $g$ with the height $v= H_{4h-1}$. From the proof of Lemma~\ref{lem: shape d odd interval} it follows that   $T_\alpha^{h}(t_1) \in (-\delta_1,r_h)$. We find that 
$$
T_\alpha^{h+1}(t_1) =T_\alpha(T_\alpha^h(t_1)) = -\lambda+\frac{\rho-1}{\lambda-\rho-1}.
$$
Since $T_\alpha^{h+1}(t_1) < r_{h+1}$, we conclude that 
$$
\begin{aligned}
T_\alpha^{2h+1}(t_1) &=(S^{-1}T)^{h} S^{-2}T (S^{-1}T)^{h-1}S^{-2}T(t_1)\\
& =  B_{h+1}\left[ \begin{array}{cc}
-1& -\lambda+1\\
\lambda-1 & \lambda^2 -\lambda-1      
\end{array} \right]   \left(-\lambda+\frac{\rho-1}{\lambda-\rho-1}\right)\\
&= \frac{\lambda-2\rho}{2+\lambda(\rho-2)} = \frac{-\rho}{1+\lambda\rho}= t_1.
\end{aligned}
$$
We find that $\left(t_1,\lambda-\rho\right)$ is a fixed point of $\mathcal{T}_\alpha^{2h+1}$.

The rest of the proof is similar to the even case for $\alpha \in \left(\frac12,\frac{1}{\lambda}\right)$. In this case $(\tau_{k},\nu_k) = T_\alpha^{-k(2h+1)} \left(l_h,\lambda-\rho\right)$. We have $\tau_{k-1}<\tau_k$ and $\displaystyle{\lim_{k\rightarrow\infty} \tau_k = t_1}$. From
$c_k= \Theta_n(\tau_{k},\nu_k)  =\frac{-\tau_{k-1}}{1+\tau_{k-1}\nu_{k-1}}$, we find $c_k<c_{k-1}$ and 
$$\lim_{k\rightarrow\infty} c_k = \frac{-1}{\frac{1}{t_1}+\lambda-\rho}= \frac{-1}{\frac{1+(\lambda-\rho)\mathcal{H}_q}{-\mathcal{H}_q}+\lambda-\rho} = \mathcal{H}_q.$$

\subsection{Odd case for \texorpdfstring{$\alpha = \frac{\rho}{\lambda}$}{alpha is rho/lambda}}
Hitoshi Nakada recently observed that the dynamical systems $(\Omega_{1/2},\mu_{1/2},{\mathcal T}_{1/2})$ and $(\Omega_{\rho/\lambda},\mu_{\rho/\lambda},{\mathcal T}_{\rho/\lambda})$ are metrically isomorphic via $M$ given in~(\ref{eq: M mirror}).  In Section~\ref{sec: even lambda} we used this isomorphism to derive results for $(\Omega_{1/\lambda},\mu_{1/\lambda},{\mathcal T}_{1/\lambda})$ from $(\Omega_{1/2},\mu_{1/2},{\mathcal T}_{1/2})$ by applying Theorem~\ref{th: sequence theta dual }. For odd $q$ and $\alpha=\frac{\rho}{\lambda}$ we can do the same and for this case the proof of Theorem~\ref{th: 3 tong odd general} is similar to the one for the even case with $\alpha=\frac{1}{\lambda}$ given in Section~\ref{sec: even lambda}. 

For $q=3$ this result (with in this case $K=1$) had been known for a long time for the nearest integer continued fraction expansion (the case $\alpha = 1/2$) and for the singular continued fraction expansion ($\alpha=\frac{1}{2}(\sqrt{5}-1)$), cf.~\cite{HK}.

\subsection{Odd case for \texorpdfstring{$\alpha \in  (\frac{\rho}{\lambda},1/\lambda]$}{alpha in the right interval}}
In this last case the natural extension $\Omega_\alpha$ is given by $\Omega_\alpha = \bigcup_{n=1}^{4h+3} J_n \times [0,H_n]$. With intervals given by
$$
\begin{aligned}
J_{2n-1} &=& [l_{n-1},r_n)& \quad \textrm{for } n=1,2,\dots ,h+1&\\
J_{2n} &=& [r_n,l_n) &\quad \textrm{for } n=1,2,\dots ,h &  \quad \textrm{ and } \quad
 J_{2h+2} = [r_{h+1},r_0),
\end{aligned}
$$

and heights defined by
$$
H_1 = \frac{1}{\lambda+1}, H_2=\frac1{\lambda} \ \quad \textrm{ and } \quad
 H_n = \frac{1}{\lambda-H_{n-2}} \quad \textrm{ for } n=3,4,\dots,2h+2.
$$
In~\cite{DKS} was shown that $H_{2h}=\lambda-1, H_{2h+1}= \frac{\lambda}{2}$ and $H_{2h+2}=1$. If $\alpha = \frac{1}{\lambda}$ the intervals $J_{2n-1}$ are empty; see Theorem~\ref{th: position l r odd}.  Again we have $l_h = \frac{\alpha \lambda-1}{1-\alpha\lambda(\lambda-1)}$. 

\subsubsection{Points in \texorpdfstring{$\mathcal{ D}^+$}{D+}}

We saw in Section~\ref{intersection} that ${\mathcal D}^+\neq \emptyset$. The leftmost point of $\overline{\mathcal{ D}^+}$ is given by $(\rho,\rho)$. Using the same techniques as in the rest of this article yields
$$
\begin{aligned}
T^{h+1}_\alpha (\rho)& =  (S^{-1}T)^{h}ST(\rho) \\
&
= B_{h+1} \left[ \begin{array}{cc}
-1  & -1\\
1 & \lambda-1      
\end{array} \right] (-\rho) = \frac{\rho-1}{-\rho+\lambda-1} =\rho.
\end{aligned}
$$
It easily follows that $ \mathcal{T}^{h+1}_\alpha(\rho,\rho) =(\rho,\rho)$.

\begin{Rmk}
In~\cite{BKS} was shown that $\rho = [\overline{+1:1,(-1:1)^h}]$ from which immediately follows $T_\alpha^{h+1}(\rho)=\rho$.

\end{Rmk}

We note that ${\mathcal T}_\alpha$ `flips' $\mathcal{ D}^+$ in the first step, in the sense that $\mathcal{T}_\alpha(\rho,\rho)$ is the rightmost point of $\mathcal{T}(\overline{D^+})$. The orientation is preserved in the next $h$ steps, so  we know that the right most point of ${\mathcal T}_\alpha^{h+1}(\overline{{\mathcal D}^+})$ is given by $\mathcal{T}_\alpha^{h+1}(\rho,\rho) = (\rho,\rho)$. Therefore we find that
$$
{\mathcal T}_\alpha^{h+1}({\mathcal D}^+) \cap {\mathcal D}^+ =\emptyset .
$$
 So after one round all points in $\mathcal{ D}^+$ are flushed out of $\mathcal{ D}$. Furthermore, it is straightforward to check that  ${\mathcal T}_\alpha({\mathcal D}^+)\subset {\mathcal D}$, and that ${\mathcal T}_\alpha^{i}({\mathcal D}^+)\cap {\mathcal T}_\alpha^{j}({\mathcal D}^+)=\emptyset$ for $0\leq i<j\leq h+1$. 


\subsubsection{Points in \texorpdfstring{$\mathcal{ D}^-$}{D-}}
Like in the previous section we only describe the part of $\mathcal{ D^-}$ on the right hand side of the line $t=-\delta_1$. 

\begin{Lemma}
If $\alpha \in \left(\frac{\rho}{\lambda},\alpha_4\right]$, then $\mathcal{ D}^-$ consists of two  components: $\mathcal{ D}_1$ bounded by the line $t=-\delta_1$ from the left, the line $v=H_{2h}=\lambda-1$ from above and the graph of $g$ from below and $\mathcal{ D}_2$ bounded by the line $t=l_h$ from the left, the line $v=H_{2h+1}=\frac{\lambda}{2}$ from above and the graph of $g$ from below. 

If $\alpha \in \left(\alpha_4, \frac{1}{\lambda}\right]$, then $\mathcal{ D}^-= \mathcal{ D}_1$.
\end{Lemma}
\begin{proof}
For all $\alpha \in \left(\frac{\rho}{\lambda},\frac{1}{\lambda}\right]$ we have that $g(-\delta_1)<H_{2h} <g(l_h)$. Furthermore $g(l_h)< H_{2h+1}$ if and only if  $\alpha < \alpha_4$.
\end{proof}

Again we must distinguish between two subcases; also see Figure~\ref{im: last odd dangerzones}.

\begin{figure}[!ht]
\includegraphics[height=90mm]{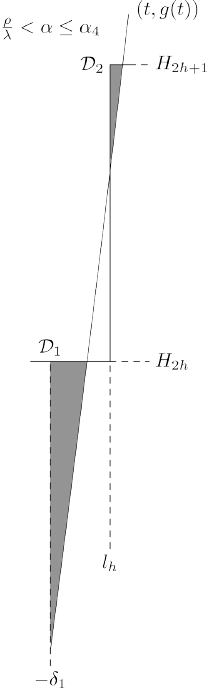}
\includegraphics[height=90mm]{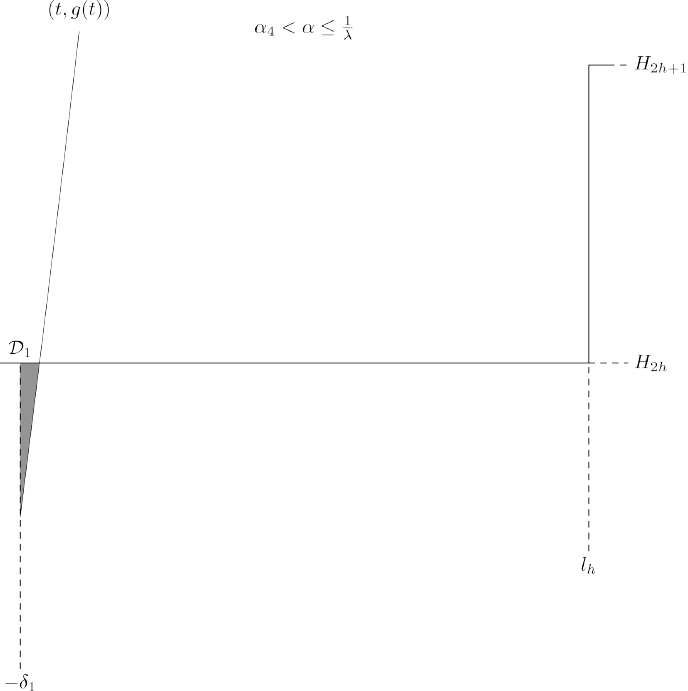}
\caption{On the left we used $\alpha = 0.50098 <\alpha_4$, on the right we use $\alpha =0.52> \alpha_4$.}
\label{im: last odd dangerzones}
\end{figure}

\emph{Proof of Theorem~\ref{th: 3 tong odd general} for $ \frac{\rho}{\lambda} <\alpha\leq \alpha_4$.} 
First we note that $T^{h+1}(-\delta_1) = l_h$ and thus after one round all points in $\mathcal{ D}_1$ are either flushed or send to $\mathcal{ D}_2$. 

We look at the orbit of $l_h$ under $T_\alpha$. We find $T_\alpha(l_h) = \frac{1+\alpha\lambda(\lambda+1)-2\lambda}{1-\alpha\lambda}$ and some calculations show that $r_1 < T_\alpha(l_h) < l_1$. We conclude that 
$$
\begin{aligned}
T_\alpha^h(l_h) &= (S^{-1}T)^{h-1}S^{-2}T(l_h) =  \left[ \begin{array}{cc}
\lambda^2-\lambda+1 &\lambda-1\\
-\lambda^3+\lambda^2+1 & -\lambda^2+\lambda+1
\end{array}\right] (l_h)\\& = \frac{2\lambda + (\alpha -1)\lambda^2 - 2}{((1-\alpha )\lambda^2 + 2\alpha +1-2\lambda )\lambda}.
\end{aligned}
$$

One can check that for $\alpha > \frac{\rho}{\lambda}$ we have $-\delta_1 < T_\alpha^h(l_h) <l_h$. We find
$$
T_\alpha^{h+1}(l_h)= T_\alpha(T^h_\alpha(l_h))=\frac{(\lambda^2-2\lambda-\alpha(\lambda^2+2)+3)\lambda }{2\lambda+(\alpha-1)\lambda^2-2} <r_1.
$$

So we finally conclude that the image of $l_h$ after one round of $2h+1$ steps is given by
$$
\begin{aligned}
T_\alpha^{2h+1}(l_h) &= (S^{-1}T)^{h}S^{-2}T(S^{-1}T)^{h-1}S^{-2}T(l_h) =  \left[ \begin{array}{cc}
\lambda^2+2 &\lambda\\
-2\lambda^2+\lambda & 2-2\lambda 
\end{array}\right] (l_h)\\
=& \frac{(-\alpha+1)\lambda^2-2\alpha\lambda-\lambda+2}{(3\alpha-2)\lambda^2+(3-2\alpha)\lambda-2}.
\end{aligned}
$$

The right end point of $\mathcal{ D}_2$ is given by $\left(\frac{-2\mathcal{H}_q}{\lambda \mathcal{H}_q+2},\frac{\lambda}{2}\right)$. Since for $ \frac{\rho}{\lambda } < \alpha \leq \alpha_4$ we have
$\frac{-2\mathcal{H}_q}{\lambda \mathcal{H}_q+2}< \frac{(-\alpha+1)\lambda^2-2\alpha\lambda-\lambda+2}{(3\alpha-2)\lambda^2+(3-2\alpha)\lambda-2}< 0$, we find that all points are flushed out of $\mathcal{ D}_2$ after one round. So all points in $\mathcal{ D}_1$ are flushed after at most $3h+2$ steps.

\emph{Proof of Theorem~\ref{th: 3 tong odd general} for $\alpha_4 < \alpha \leq \frac{1}{\lambda}$.} 
All points are flushed out of $\mathcal{ D}_1$ after $h+1$ steps, since the line $t=-\delta_1$ is mapped by $T_\alpha$ to the line $t=l_0$. So after $h+1$ steps all points in $\mathcal{ D}$ are mapped to points that lie on the right hand side of the line $t=l_h$ and thus outside of $\mathcal{ D}_1$. 

We need to check that $\mathcal{T}_\alpha^{h+1}(\mathcal{ D}_1) \cap \mathcal{ D}^+ = \emptyset$. The rightmost vertex of $\overline{\mathcal{ D}_1}$ is given by $\left( \displaystyle{\frac{-1}{\lambda-1+\frac{1}{\mathcal{H}_q}}}, \lambda -1 \right)$.
With similar techniques as before we find that applying $(h+1)$-times  ${\mathcal T}_{\alpha}$ to the rightmost vertex of ${\mathcal D}_1$ yields as first coordinate
$$
\begin{aligned}
&T_\alpha^{h+1} \left(\frac{-1}{\lambda-1+\frac{1}{\mathcal{H}_q}} \right)
	\\ &= \left[ \begin{array}{cc}
	-\lambda  &-1\\
	1 & 0
	\end{array}\right]  \left[ \begin{array}{cc}
	\lambda^2-\lambda+1 &\lambda-1\\
	-\lambda^3+\lambda^2+1 & -\lambda^2+\lambda+1
	\end{array}\right]\left(\frac{-1}{\lambda-1+\frac{1}{\mathcal{H}_q}} \right)\\
	&= \frac{1-2\mathcal{H}_q}{1+(\mathcal{H}_q-1)\lambda} <\rho.
\end{aligned}
$$

This completes the proof. \qed

\section{Borel and Hurwitz constants for \texorpdfstring{$\alpha$}{alpha}-Rosen fractions}
\label{sec: Hurwitz rosen} 
In this section we first prove the Borel-type Theorem~\ref{thm: article Borel alpha} and then derive a Hurwitz-type result for certain values of $\alpha$.

\subsection{Borel for \texorpdfstring{$\alpha$}{alpha}-Rosen fractions}

%

Let $q \geq 3$ be an integer. Recall that for all $\alpha \in \left[\frac12,\frac{1}{\lambda}\right] $ and every $G_q$-irrational $x \in [(\alpha-1)\lambda,\alpha\lambda)$ the future $t_n$ and past $v_n$ satisfy
\begin{equation}
\label{eq: tn end equation}
\mathcal{T}^n_\alpha (x,0) =(t_n,v_n)  \quad \textrm{for all } n \geq 0.
\end{equation}

We denote by $r$ the number of steps in a round, so $r=p-1$ for even $q$ and $r=2h+1$ for odd $q$. Furthermore, for even $q$ we define $\rho=1$. 

\begin{Lemma}\label{lem:FixedAndFlushed}   
Let either $q$ be even and $\alpha \in \left[ \frac12, \frac{1}{\lambda}\right]$ or $q$ be odd and $\alpha \in \left[\frac12,\frac{\rho}{\lambda}\right]$. Let $\mathcal F$ denote the
fixed point set in ${\mathcal D}$ of ${\mathcal T}_{\alpha}^{r}$. Then
\begin{enumerate}
\item[($i$)]  $\mathcal F = \left\{\, {\mathcal T}_{\alpha}^i\left(\frac{-\rho}{1+\lambda\rho} ,\lambda -\rho\right)\,\bigg|\, i=0,1,\dots,r -1 \right\}$;

\item[($ii$)]     For every  $x$ and every $n\geq 0$, $(t_n,v_n) \notin \mathcal F$;

\item[($iii$)]  For every $G_q$-irrational number $x$ there are infinitely many $n$ for which $(t_n,v_n)\notin {\mathcal D}$;

\item[($iv$)]  For each $i=0,1,\dots,r-1$, let  $x_i=T_{\alpha}^i\left(\frac{-\rho}{1+\lambda\rho}\right)\,$.   Then for all $n\geq 0$,  ${\mathcal T}_{\alpha}^n(x_i,0)\notin {\mathcal D}$.
However, ${\mathcal T}_{\alpha}^{kr}(x_i,0)$ converges from below on the vertical line $x=x_i$ to ${\mathcal T_{\alpha}}^i\left(\frac{-\rho}{1+\lambda\rho} ,\lambda -\rho\right)$.
\end{enumerate}
\end{Lemma}
\begin{proof}
Assume that $q$ is even and $\alpha \in \left(\frac12,\frac{1}{\lambda}\right)$, the other cases can be proven in a similar way. In this case $r=p-1$ and we have by Corollary~\ref{cor: even fixed point} that $\left(\frac{-1}{\lambda+1},\lambda-1 \right)$ is a fixed point of $\mathcal{T}_\alpha^{p-1}$. It follows that points in the $\mathcal{T}_\alpha$-orbit of $\left(\frac{-1}{\lambda+1},\lambda-1 \right)$ must also be fixed points of $\mathcal{T}_\alpha^{p-1}$, which proves ($i$). In each fixed point of $\mathcal{T}^{p-1}_\alpha$ both of the coordinates have a periodic infinite expansion. For ($ii$) we note that by definition $v_n$ has a finite expansion of length $n$ and therefore for every $x$ and every $n$ we have $v_n \notin \mathcal F$. We conclude from the section on flushing on page \pageref{sec: flush}, that for every $G_q$-irrational number $x$ there are infinitely many $n$ for which $(t_n,v_n)\notin {\mathcal D}$, which is ($iii$). Finally,  for each $i=0,1,\dots,r-1$ and every $n\geq 0$ the points ${\mathcal T}_{\alpha}^{n}(x_i,0)$ are below the graph of $f$. The fixed points ${\mathcal T_{\alpha}}^i\left(\frac{-1}{1+\lambda} ,\lambda -1\right)$ are attractors for these points, cf. the proof of Lemma~\ref{lem: d flushen}.
\end{proof}

\emph{Proof of Theorem~\ref{thm: article Borel alpha}.} Let $q \geq 3$ be an integer and let $x$ be a $G_q$-irrational.  We first assume that we do not have a finite spectrum, so if $q$ is even, we consider all $\alpha \in \left[ \frac12, \frac{1}{\lambda}\right]$ and if $q$ is odd  we assume $\alpha \in  \left[\frac12,\frac{\rho}{\lambda}\right]$. 

From~Lemma~\ref{lem:FixedAndFlushed} (ii) we see that $(t_n,v_n)$ can never be a fixed point for any $n \geq 0$. From (iii) we know that here are infinitely many $n$ for which $(t_n,v_n)\notin {\mathcal D}$, so there are infinitely many $n \in \N$ for which
\[ 
q_n^2 \left| x - \frac{p_n}{q_n} \right| \leq \mathcal{H}_q. 
\]
It remains to show that also in this case $\mathcal{H}_q$ can not be replaced by a smaller constant. Take $x$ such that $t_1 = \frac{-\rho}{1+\lambda\rho}$. By definition of $\Omega_\alpha$ we know that $v_1 \leq \lambda-1$ and since $v_1$ has a finite expansion we find $v_1 < \lambda-1$. For all $l \geq 1$ we have $(t_{1+r\,l},v_{1+r\,l}) \notin \mathcal{ D}$ and for every $0 \leq i<r$ one has $\displaystyle{\lim_{l\rightarrow\infty}}\mathcal{T}_\alpha^i(t_{1+r\,l},v_{1+r\,l}) = \mathcal{T}_\alpha^i\left(\frac{-\rho}{1+\lambda\rho} ,\lambda -\rho\right)$. So, $\mathcal{H}_q$ can not be replaced by a smaller constant. 

Finally, assume $q$ is odd  and  $\alpha \in \left(\frac{\rho}{\lambda},\frac{1}{\lambda}\right]$. From the finite spectrum in Theorem~\ref{th: 3 tong odd general} it immediately follows that in this case there are  infinitely many $n$ for which
\[ 
q_n^2 \left| x - \frac{p_n}{q_n} \right| \leq \mathcal{H}_q. 
\]
It remains to show that in this case the constant $\mathcal{H}_q$ cannot be replaced by a smaller constant. Consider $(\rho,\rho)$, the fixed point of $\mathcal{T}_\alpha^{h+1}$. We find
$$
\mathcal{T}_\alpha(\rho,\rho) = \left(\frac{1}{\rho}-\lambda,\frac{1}{\lambda+\rho}\right).
$$
Note that $H_1=\frac{1}{\lambda+1}<\frac{1}{\lambda+\rho}<\frac{1}{\lambda}=H_2$. Furthermore
$$
f\left(\frac{1}{\rho}-\lambda\right)= \frac{\mathcal{H}_q}{1-\mathcal{H}_q\left(\frac{1}{\rho}-\lambda\right)} = \frac{\rho}{\rho^2+1-\rho\left(\frac{1}{\rho}-\lambda\right)} = \frac{1}{\lambda+\rho}.
$$
So $\mathcal{T}_\alpha(\rho,\rho)$ lies on the graph of $f$ bounding $\mathcal{ D}$. Points in $\mathcal{ D}$ on the left hand side of $x=-\delta_1$ are mapped into $\mathcal{ D}$ by $\mathcal{T}_\alpha$ and we find $\mathcal{T}^2_\alpha(\rho,\rho), \mathcal{T}^3_\alpha(\rho,\rho),\dots,\mathcal{T}^{h-1}_\alpha(\rho,\rho)$ all are in $\mathcal{ D}$. 

It is easy to check that $\mathcal{T}_\alpha\left(\frac{-1}{\lambda+\rho},\lambda-\frac{1}{\rho}\right) = (\rho,\rho)$ and that $g\left(\frac{-1}{\lambda+\rho}\right)=\lambda-\frac{1}{\rho}$. We conclude that $\mathcal{T}^h_\alpha(\rho,\rho)= \mathcal{T}_\alpha^{-1}(\rho,\rho)=\left(\frac{-1}{\lambda+\rho},\lambda-\frac{1}{\rho}\right)$ lies on the graph of $g$.

Now consider any point $(t_n,v_n)=(\rho,y)$, since $\rho$ has a periodic infinite expansion we know $y\neq \rho$. However the periodic orbit of $(\rho,\rho)$ is an attractor of the orbit of $(\rho,y)$, so $\displaystyle{\lim_{k\rightarrow \infty} \mathcal{T}_\alpha^{k(h+1)} (\rho,y) = (\rho,\rho)}$. It follows that the constant $\mathcal{H}_q$ is best possible in this case.
\hfill $\Box$\medskip\

\subsection{Hurwitz for \texorpdfstring{$\alpha$}{alpha}-Rosen fractions}
For odd $q$ and some values of $\alpha$ we can generalize Theorem~\ref{thm: article Borel alpha} to a Hurwitz-type theorem, which is the Haas-Series result mentioned in Section~\ref{sec: rosen intro} From~\cite{BJW} it follows that for all $\alpha \in \left[\frac12,\frac{1}{\lambda}\right]$, for all $z \geq 0$ and for almost all $x$, the limit
$$
\lim_{n \rightarrow \infty} \frac{1}{n} \# \{ 1 \leq j \leq n | \Theta_j (x) \leq z \}
$$
exists and equals the distribution function $F_\alpha$, which satisfies
$$
F_\alpha(z) = \bar{\mu} \left(\left\{ (t,v) \in \Omega_\alpha \left| v \leq f(t) = \frac{z}{1-zt} \right. \right\}\right),
$$
where $\bar{\mu}$ is the invariant measure for $\mathcal{T}_\alpha$ given in~\cite{DKS}.
Defining the Lenstra constant $\mathcal{L}_\alpha$ by
\begin{equation}
\label{eq:lenstra even }
\mathcal{L}_a = \max \left\{ c>0 \left| \left (t,\frac{c}{1-ct} \right) \in \Omega_\alpha \right. ,\textrm{ for all } t \in [l_0,r_0] \right\}.
\end{equation}
As mentioned in Section~\ref{sec: intro lenstra hurwitz}, the distribution function $F_\alpha$ is a linear map with positive slope for $z \in [0,\mathcal{L}_\alpha]$. Nakada showed in~\cite{N2} that the Lenstra constant is equal to the Legendre constant whenever the latter constant exists. In his article he particularly mentioned Rosen fractions and $\alpha$-expansions, but this result also holds for $\alpha$-Rosen fractions. So if $p/q$ is a $G_q$-rational and
$$
q^2 \left| x - \frac{p}{q} \right| < \mathcal{L}_\alpha,
$$
then $p/q$ is an $\alpha$-Rosen convergent of $x$. 

Since for the standard Rosen fractions (where $\alpha = \frac12$) one has that $\mathcal{L}_\alpha<\mathcal{H}_q$, the Haas-Series result does not follow from Theorem~\ref{thm: article Borel alpha}; see also the discussion on the results of Legendre, Borel and Hurwitz in Section~\ref{sec: intro lenstra hurwitz}. 

One wonders whether $\alpha$-Rosen fractions could yield a continued fraction proof of the Haas-Series result for particular values of $\alpha$. Proposition 4.3 of~\cite{DKS} states that for even $\alpha$-Rosen-fractions
$$
\mathcal{L}_a = \min \left\{ \frac{\lambda}{\lambda+2},\frac{\lambda(2-\alpha \lambda^2)}{4-\lambda^2} \right\}.
$$

Since $\mathcal{L}_\alpha <\mathcal{H}_q = \frac12$, we see that a direct continued fraction proof of a Hurwitz-result cannot be given in this case. In~\cite{DKS} the more involved formula for $\mathcal{L}_\alpha$ for odd $\alpha$-Rosen fractions was not given. For odd $q$ we have the following proposition.

\begin{Proposition}
\label{lem: lenstra constant}
Let $q \geq 3$ be an odd integer and let $\alpha_L = \frac{\mathcal{H}_q}{\lambda(1-\mathcal{H}_q)}$. Then $\mathcal{L}_\alpha  < \mathcal{H}_q$ for $\alpha \in [1/2,\alpha_L)$, while 
$\mathcal{L}_\alpha = \frac{\alpha\lambda}{\alpha\lambda+1} > \mathcal{H}_q$ for $\alpha \in [\alpha_L,1/\lambda]$.
\end{Proposition}
\begin{proof}
For every $\alpha \in [1/2,\alpha_L)$ there is a $C< \mathcal{H}_q$ such that $ \left(t,\frac{C}{1-Ct} \right) \notin \Omega_\alpha$. We only prove it for $\alpha=\rho/\lambda$. Consider the point $\left(\rho, \frac{C}{1-C\rho}\right)$. This point is in $\Omega_\rho$ if the $y$-coordinate is smaller than top height $\frac{\lambda}{2}$. We have
$$
\frac{C}{1-C\rho} < \frac{\lambda}{2} \textrm{ if and only if } C < \frac{\lambda}{\lambda\rho+2}. 
$$
So by~(\ref{eq:lenstra even }) we have that $\mathcal{L}_\alpha = \frac{\lambda}{\lambda\rho+2} < \mathcal{H}_q$. 

Let $\alpha \in  [\alpha_L,1/\lambda]$. Consider the point $\left(r_1, \frac{C}{1-Cr_1} \right)$, this point is in $\Omega_\alpha$ if the $y$-coordinate is smaller than $H_1=\frac{1}{\lambda+1}$. Using $r_1 = \frac{1}{\alpha\lambda}-\lambda$ we find that
$
\frac{C}{1-C r_1} \leq \frac{1}{\lambda+1} \textrm{ if and only if } C \leq \frac{\alpha\lambda}{\alpha\lambda+1}. 
$ 
For $C = \frac{\alpha\lambda}{\alpha\lambda+1} $ we easily find that all points $\left (t,\frac{C}{1-Ct} \right) $ with $t \in (l_0,r_0)$ are in $\Omega_\alpha$, so $\mathcal{L}_\alpha = \frac{\alpha\lambda}{\alpha\lambda+1}$. Finally we see that $\mathcal{L}_\alpha > \mathcal{H}_q$ if and only if $ \alpha > \alpha_L.$
\end{proof}

Thus the Hurwitz-type theorem of Haas-Series follows from Theorem~\ref{thm: article Borel alpha}, Proposition~\ref{lem: lenstra constant} and Nakada's result from~\cite{N2} in case $q$ is odd and $\alpha \in [\alpha_L,1/\lambda]$.

\bibliography{bibionica}
\bibliographystyle{abbrv} 

\end{document}